\documentclass[a4paper, 10pt, reqno,final]{article}

\usepackage[T1]{fontenc}		     
\usepackage[utf8]{inputenc}			 
\usepackage[english]{babel}

\usepackage{amsmath}
\usepackage{amssymb}
\usepackage{amsthm}
	\theoremstyle{plain}
		\newtheorem{mainthm}{\textsc{Theorem}}		

		\newtheorem{thm}{Theorem}[section]	
		\newtheorem{cor}[thm]{Corollary}	
		\newtheorem{maincor}{\textsc{Corollary}}		%

		\newtheorem{lem}[thm]{Lemma}		
		\newtheorem{prop}[thm]{Proposition}

	\theoremstyle{definition}
		\newtheorem{defn}[thm]{Definition}	

	\theoremstyle{remark}
		\newtheorem{rem}[thm]{Remark}		
		\newtheorem{note}[thm]{Notation}		

\numberwithin{equation}{section}	

\usepackage{mathtools}		
\usepackage{mathrsfs}		
\usepackage{eucal}			

\usepackage{braket}			
\usepackage[all,pdf]{xy}		
\usepackage{mparhack}		
\usepackage{relsize}			

\usepackage{a4wide}		

\usepackage{booktabs}		
\usepackage{multirow}		
\usepackage{caption}		
\usepackage{rotating}
\usepackage{subfig}

\usepackage{varioref}		
\usepackage{footmisc}

\usepackage{graphicx}		
\usepackage{epsfig}

\usepackage{enumerate}		
\usepackage[colorlinks=true, linkcolor=black, citecolor=black,
urlcolor=blue]{hyperref}		
\mathtoolsset{showonlyrefs}
\newcommand{\GL}{\mathrm{GL}}

\newcommand{\trasp}[1]{{#1}^\mathsf{T}}	
					
\newcommand{\EE}{\mathbf{E}}		
\newcommand{\traspm}[1]{{#1}^\mathsf{-T}}		
\newcommand{\iMor}{\mathrm{n_-}}		
\newcommand{\iMorse}{\iota_{\scriptscriptstyle{\mathrm{Mor}}}}
\newcommand{\R}{\mathbb{R}}	

\newcommand{\C}{\mathbb{C}}		
\newcommand{\T}{\mathbb{T}}		
\newcommand{\U}{\mathbb{U}}		
\newcommand{\OO}{\mathrm{O}}		
\newcommand{\spec}{\sigma}		
\newcommand{\SO}{\mathrm{SO}}

\newcommand{\irel}{I}		
\newcommand{\ispec}{\iota_{\textup{spec}}}				
\newcommand{\igeo}{\iota_{\textup{geo}}}		
\newcommand{\Sp}{\mathrm{Sp}}

\newcommand{\Lagr}{\Lambda}

\newcommand{\Ddt}{\nabla_t}

\newcommand{\dt}{\tfrac{\mathrm{d}}{\mathrm{d}t}}

\newcommand{\rie}[2]{\langle{#1}, {#2} \rangle_g}

\newcommand{\Graph}{\mathrm{Gr\,}}

\newcommand{\Gr}{\mathrm{Gr}}

\newcommand{\im}{\mathrm{rge}\,}

\newcommand{\ind}{\mathrm{ind}\,}

\newcommand{\norm}[1]{\left\| #1 \right\|}			
				
\newcommand{\traspinv}[1]{{#1}^\mathsf{-T}}				%

\newcommand{\coiMor}{\coindex}

\newcommand{\M}{\mathcal{M}}	%
\newcommand{\N}{\mathbb{N}}		

\newcommand{\iCLM}{\iota_{\scriptscriptstyle{\mathrm{CLM}}}}

\newcommand{\iomega}[1]{\iota_{#1}}
\newcommand{\Z}{\mathbb{Z}}		

\newcommand{\coindex}{\mathrm{n_+}}
\newcommand{\iiindex}{\mathrm{n_-}}

\newcommand{\noo}[1]{\overset {\mbox{%
\lower1pt\hbox{${\scriptscriptstyle o}$}}}n^{\mbox{%
\lower2pt\hbox{$\scriptscriptstyle #1$}}}}

\newcommand{\ptl}{P(L)}
\newcommand{\pth}{P(H)}

\DeclareMathOperator{\spfl}{sf}			
\DeclareMathOperator{\sgn}{sgn}		
\DeclareMathOperator{\rk}{rank}		

\renewcommand{\leq}{\leqslant}
\renewcommand{\geq}{\geqslant}

\renewcommand{\tilde}{\widetilde}
\renewcommand{\=}{\coloneqq}			

\newcommand{\email}[1]{\href{mailto:#1}{\textsf{#1}}}

\usepackage{bm}

\newcommand{\Id}{I}

\title{Linear instability  for  periodic orbits of non-autonomous Lagrangian systems}
\author{Alessandro Portaluri
\thanks{The
author is partially supported by Prin 2015 ``Variational methods, with applications to
problems in mathematical physics and geometry” No.~$\mathrm{2015KB9WPT\_001}$.}, Li Wu, Ran Yang }
\date{\today}

\date{\today}
\begin{document}
 \maketitle

\begin{abstract}

Inspired by the classical Poincaré criterion about the instability of  orientation preserving minimizing closed geodesics on surfaces,  we investigate  the relation intertwining the instability and the variational properties of periodic solutions of a non-autonomous Lagrangian  on a finite dimensional Riemannian  manifold.

We establish a general criterion for a priori detecting  the linear instability of a periodic orbit on a Riemannian manifold for a (maybe not Legendre convex) non-autonomous Lagrangian simply by looking at the parity of the spectral index, which is the right substitute of the Morse index in the framework of strongly indefinite variational problems and defined in terms of the spectral flow of a path of Fredholm quadratic forms on a Hilbert bundle.

\end{abstract}


\section{Introduction }\label{sec:intro}

A celebrated result proved by Poincaré in the beginning of the last century put on evidence the relation between  the (linear and exponential) instability of an orientation preserving  closed geodesic as a critical point of the geodesic energy functional on the free loop space on a surface and the minimization property of such a critical point. The literature on this criterion is quite broad. We refer the interested reader to \cite{BJP14, BJP16, HPY19}  and references therein. Loosely speaking, a closed geodesic  is termed linearly stable if the monodromy matrix associated to it  splits into two-dimensional rotations. Accordingly, it is diagonalizable and all Floquet multipliers belong to the unit circle of the complex plane. Moreover, if $1$ is not a Floquet multiplier, we term the geodesic non-degenerate. In 1988 the Poincaré instability criterion for closed geodesics was generalized in several interesting directions by Treschev in \cite{Tres88, BT15} and references therein.

Several years later, the authors in \cite{HS10} proved a generalization of the aforementioned result, dropping the non-degeneracy assumption. Very recently, authors  in \cite{HPY19}, by using a mix of variational and symplectic techniques, were able to establish, among others,  a general criterion in order to detect the (linear) instability of a closed geodesic on a finite dimensional semi-Riemannian manifold, by controlling the parity of an integer which naturally replace, in this setting, the classical Morse index (which is, in general, infinite). In fact, in the semi-Riemannian (not Riemannian world) the critical points of the geodesic energy functional have, in general, an infinite Morse index and co-index. However, in this strongly indefinite situation a natural substitute of the Morse index is represented by a topological invariant known in literature as the {\em spectral flow\/}.

A natural question (which somehow was the main motivation for writing this paper) is how the results proved by authors in \cite{HPY19} can be carried over in the case of non-autonomous (not Legendre convex) Lagrangian functions on Riemannian manifolds.  

Dropping the Legendre convexity, is quite a challenging task and somehow reflect the sort of indefiniteness behavior  appearing in the truly semi-Riemannian world, but it is crucial in the investigation of many indefinite variational problems naturally arising in the applications. In fact, as already observed, critical points have in general infinite Morse index and co-index. However, as in the (classical) Legendre convex case, under some suitable growing  assumption on the Lagrangian function,  the second variation gives arise to a Fredholm (strongly indefinite) quadratic form. Taking advantage of this,  we associate to each periodic orbit and index called {\em spectral index\/} $\ispec$ and defined in terms of the  {\em spectral flow\/}. 

 In general the spectral flow depends on the homotopy class of the whole path and not only on its ends. However in the special case of geodesics on semi-Riemannian manifolds, things are simpler since it depends only on the endpoints of the path and therefore it can be considered a relative form of Morse index known in literature as relative Morse index.
 
It is well-known that this invariant is strictly related to a symplectic invariant known in literature as Maslov index (which is an intersection invariant constructed in the Lagrangian Grassmannian manifold of a symplectic space) that plays a crucial role in detecting the stability properties of a periodic orbit.  (We refer the interested reader to see \cite{APS08, MPP05, MPP07, HS09, Por08, RS95} and references therein). Very recently, new spectral flow formulas has been established and applied in the study, for detecting bifurcation of heteroclinic and homoclinic orbits of Hamiltonian systems or bifurcation of semi-Riemannian geodesics. (Cfr. \cite{PW16,MPW17, BHPT19, HP17, HP19, HPY17}).

As in the geodesic case the stability as well as the spectral index both  {\em strongly depend\/} on the trivialization of the pull-back bundle along the orbit. This is not because we are dealing with not Legendre convex Lagrangian functions or on semi-Riemannian manifold or even because we are dealing with orientation and not-orientation preserving  (maybe degenerate) periodic orbits. This very much depends on the topological property of the manifolds. A first attempt to find a general criterion which is independent on the chosen trivialization, could be to restrict to contractible periodic orbits. In this way it is possible to construct a trivialization of the pull-back of the tangent bundle along the orbit by filling in a disk. However, as one can expect, this is not enough to solve the problem! This in fact happens precisely if we are in the presence of spheres $\iota: S^2 \hookrightarrow M$ with non-vanishing first Chern class.  Thus,  if the first Chern class $c_1(M)$ has vanishing integral on every $2$-sphere of M, (namely on the second homotopy group of $M$) then certainly the geometrical index will be independent on such a trivialization, since in this case any two trivializations in the chosen  class differ by a loop in the symplectic group which is homotopically trivial. However, the stability properties of a periodic orbit are somehow more sensitive on the trivialization depending not in a homotopically theoretical way on the spectrum of the monodromy matrix.


\subsection{Description of the problem and main result}

Let $T>0$, $(M,\rie{\cdot}{\cdot})$ be a (not necessarily compact or connected) smooth $n$-dimensional Riemannian manifold without boundary, which represents the configuration space of a Lagrangian dynamical system.  Elements  in the tangent bundle $TM$ will be  denoted by $(q,v)$, with $q \in M$ and  $v \in T_qM$. Let $\T\subset \R^2$ denote a circle in $\R^2$ of length $T=|\T|$ which can be identified with $\R/T\Z$, where $T\Z$ denotes the lattice generated by $T \in \R$ and  
 Let  $L: \T \times TM\to \R$  be a smooth non-autonomous (Lagrangian) function  satisfying the following assumptions
\begin{itemize}
\item[{\bf(N1)\/}] $L$ is  non-degenerate on the fibers of $TM$, that is, for every $(t,q,v)\in \T \times TM$ we have that  $\nabla_{vv}L(t,q,v)\neq 0$ is non-degenerate as a quadratic form
\item[{\bf(N2)\/}] $L$ is {\em exactly quadratic\/}  in the velocities meaning that the function $ L(t,q,v)$ is a polynomial of degree at most $2$ with respect to $v$.
\end{itemize}
\begin{rem}
 Before stating our main result, we observe that we require that the Lagrangian function is {\em exactly quadratic\/} in the velocity assumption (N2). The smoothness property of $L$ are in general not sufficient for guaranteeing the twice Fréchét differentiability of the Lagrangian action functional. In fact, the functional it is twice Fréchét differentiable if and only if $L$ is {\em exactly\/} quadratic in the velocities (meaning that for every $(t,q)$ the function $v\mapsto L(t,q,v)$ is a polynomial of degree at most $2$). In this case, the Lagrangian action functional is actually smooth. This fact it depends upon the differentiability properties of the Nemitsky operators and it is very well explained in \cite{AS09}.
 \end{rem}

\begin{mainthm}\label{thm:instability theorem for non-autonomous case}
Let $(M,\rie{\cdot}{\cdot})$ be a $n$-dimensional Riemannian manifold,  $L: \T \times TM\to \R$  be a smooth non-autonomous (Lagrangian) function  satisfying the following assumptions (N1) and (N2) and let  $x$ be a $T$-periodic orbit of the (associated) Lagrangian system. If one of the following two alternatives hold
\begin{itemize}
\item[]{\bf (OR)} $x$   is  orientation preserving and $\ispec(x) +n$   is odd
\item[]{\bf (NOR)} $x$   is non orientation preserving and $\ispec(x) +n$  is even
\end{itemize}
then $x$ is linearly unstable.
\end{mainthm}
We let
 \begin{itemize}
\item[{\bf(L1)\/}] $L$ is  $\mathscr C^2$ strictly convex  on the fibers of $TM$, that is, for every $(t,q,v)\in \T \times TM$ we have that  $\nabla_{vv}L(t,q,v)> 0$ as a quadratic form
\end{itemize}
As a direct corollary of the previous result, we get. 
\begin{maincor}\label{cor:instability for P positive}
Let $(M,\rie{\cdot}{\cdot})$ be a $n$-dimensional Riemannian manifold,  $L: \T \times TM\to \R$  be a smooth non-autonomous (Lagrangian) function  satisfying the following assumptions (L1) and (N2) and let  $x$ be a $T$-periodic orbit of the (associated) Lagrangian system. If one of the following two alternatives hold
\begin{itemize}
\item[]{\bf (OR)} $x$   is  orientation preserving and $\iMorse(x) +n$   is odd
\item[]{\bf (NOR)} $x$   is non orientation preserving and $\iMorse(x) +n$  is even
\end{itemize}
then $x$ is linearly unstable.
\end{maincor}
This result is a straightforward corollary of Theorem \ref{thm:instability theorem for non-autonomous case} and in fact follows by observing that under the assumption (L1) the spectral index reduces to the classical Morse  index. 
\begin{rem}
	Here we would like to emphasize on a couple of remarks. First of all, in the special case in which the configuration space $M$ is just the Euclidean space $\R^n$, the proof of Theorem \ref{thm:instability theorem for non-autonomous case} it is much easier simply because no trivializations are involved in this case. The reader should be aware on the crucial importance of the trivialization that somehow encodes the geometrical part of the problem. We should note in fact that the second variation (and hence the spectrum of the monodromy matrix and consequently the linear stability) will be strongly influenced upon the choice of this trivialization. (Cfr. \cite[Appendix]{AF07}).
	Even the relation between the Morse index and the Conley Zehnder index, as already discussed, clearly depends upon the chosen trivialization. Moreover, if the Lagrangian function is Legendre convex and the critical point is non-degenerate, then the spectral index reduces to the classical Morse index which is easier to be treated. 
	
	The second remark is about on the which seems completed overlooked in many papers, is related to the growth assumptions with respect to the velocity $v$ of the Lagrangian function. In fact, is precisely under assumption (N1) on the smooth Lagrangian $L$ that the Lagrangian action functional is $\mathscr C^2$ (actually smooth). This fact relies on the differentiability of the Nemitski operators and it is crucial in order to associate a Morse index (or a spectral index in the strongly indefinite case) to a critical point. 
			\end{rem}
The paper is organized as follows

\tableofcontents

\subsection*{Notation}
	At last, for the sake of the reader, let us introduce some common notation that we shall use henceforth throughout the paper. 
\begin{itemize}
\item $(M,\rie{\cdot}{\cdot})$ denotes a Riemannian manifold without boundary; $TM$ its  tangent bundle and  $T^*M$ its cotangent bundle.
\item $\Lambda^1(M)$ is the Hilbert manifold of loops on manifold $M$ having Sobolev regularity $H^1$
	\item $\omega$ denotes the symplectic structure $J$ the standard  symplectic matrix 
	\item $\iMorse(x)$ stands for the Morse index of $x$, $\ispec(x)$ for the spectral index of $x$, $\igeo(x)$ for the geometrical index index of $x$, $\iomega{1}$ denotes the Maslov-type index or Conley-Zehnder index of a symplectic matrix path, $\iCLM$ denotes the Maslov (intersection) index and finally $\spfl$ denotes the  spectral flow
	\item $P(L)$ denotes the set of $T$-periodic solutions of the Euler-Lagrange Equation ; $P(H)$ denotes  the set of $T$-periodic solutions of the Hamiltonian Equation
	\item $\delta_{ij}$ is the Kronecker symbol.  $\Id_X$ or just $\Id$ will denote the identity operator on a space $X$ and we set for simplicity  $\Id_k := \Id_{\R^k}$ for $k \in \N$.  $\Gr(\cdot)$ denotes the  graph of  its argument; $\Delta$ denotes the graph of identity matrix $\Id$
	\item $\U$ is  the unit circle of the complex plane
	\item $\OO(n)$ denotes the orthogonal group $\Sp(2n,\R)$ or just $\Sp(2n)$ denotes the $2n\times 2n$ real symplectic group
	\item $\mathfrak P$ denotes the  linearized Poincar\'{e} map
	\item $\mathcal{BF}^{sa}$ denotes  the set of all bounded	selfadjoint Fredholm operators,  $\sigma(\cdot )$ denotes the spectrum of the operator in its argument
	\item We denote throughout by the symbol $\trasp{\cdot}$ (resp. $\traspinv{\cdot}$) the transpose (resp. inverse transpose) of the operator $\cdot$. Moreover $\im(\cdot), \ker(\cdot)$ and $\rk(\cdot)$ denote respectively the image, the kernel and the rank of the argument
		\item $\Gamma$ denotes the crossing form and  $\coiMor$/$\iMor$ denote respectively the dimensions  of the positive/negative  spectral  spaces and finally $\sgn(\cdot)$ is  the  signature of the quadratic form (or matrix) in its argument and it is  given by $\sgn(\cdot)=\coiMor(\cdot)-\iMor(\cdot)$ 
	\end{itemize}

\subsection*{Acknowledgements}
The first name author wishes to thank all faculties and staff of the Mathematics Department at the Shandong University (Jinan) as well as the Department of Mathematics and Statistics at the Queen's University (Kingston)  for providing excellent working conditions during his stay.


\section{Lagrangian dynamics and Variational framework}\label{sec:variational-framework}

This section is devoted to fix our basic notation about the Lagrangian and Hamiltonian dynamics and to set up the  variational framework of the problem. 


\subsection{Lagrangian and Hamiltonian dynamics}\label{subsec:L-H-dynamics}

Let $T>0$, $(M,\rie{\cdot}{\cdot})$ be a (not necessarily compact or connected) smooth $n$-dimensional Riemannian manifold without boundary, which represents the configuration space of a Lagrangian dynamical system and we denote by $\norm{\cdot}$ the Riemannian norm.   Elements  in the tangent bundle $TM$ will be  denoted by $(q,v)$, with $q \in M$ and  $v \in T_qM$. The metric $\rie{\cdot}{\cdot}$  induces a metric on $TM$, Levi-Civita connections both on $M$ and $TM$ as well as the isomorphisms 
\begin{equation}
	T_{(q,v)}TM= T_{(q,v)}^hTM\oplus T_{(q,v)}^vTM\cong T_qM \oplus T_qM,
\end{equation}
where $T^v_{(q,v)}TM= \ker D\tau(q,v)$ and  $\tau: TM \to M$ denotes the canonical projection. 
\begin{note}
We shall denote by  $\Ddt$ the covariant derivative  of a vector field along a smooth curve $x$ with respect to the metric $\rie{\cdot}{\cdot}$. $\nabla_q$ (resp. $\nabla_v$) denotes the partial derivative along the horizontal part  (resp.  vertical part) given by the Levi-Civita connection in the splitting of $TTM$ and We shall denote by $\nabla_{vv}, \nabla_{qv}, \nabla_{qq}$ the  components of the Hessian in the splitting of $TTM$. 
\end{note}
Let $\T\subset \R^2$ denote a circle in $\R^2$ of length $T=|\T|$ which can be identified with $\R/T\Z$, where $T\Z$ denotes the lattice generated by $T \in \R$. 
 Let  $L: \T \times TM\to \R$  be a smooth non-autonomous (Lagrangian) function  satisfying the following assumptions
\begin{itemize}
\item[{\bf(N1)\/}] $L$ is  non-degenerate on the fibers of $TM$, that is, for every $(t,q,v)\in \T \times TM$ we have that  $\nabla_{vv}L(t,q,v)\neq 0$ is non-degenerate as a quadratic form
\item[{\bf(N2)\/}] $L$ is {\em exactly quadratic\/}  in the velocities meaning that the function $ L(t,q,v)$ is a polynomial of degree at most $2$ with respect to $v$.
\end{itemize}

 By \cite[Theorem 2.4.3]{Fat08}, assumption (N1) implies that $L$ defines  a unique time-dependent smooth vector field $Y_L$ on $ TM$ such that the orbits of $Y_L$ are precisely of the form  $t\mapsto \big(x(t), x'(t)\big)$, where $x \in \mathscr C^2([0,T], M)$ solves the second order ODE
\begin{equation}\label{eq:EL}
\Ddt \Big(\nabla_v L\big(t, x(t), x'(t)\big)\Big)	= \nabla_q L\big(t, x(t), x'(t)\big), \qquad t \in (0,T).
\end{equation}
Let us now consider $T^*M$, the  cotangent bundle of $M$ with canonical projection  $\tau^*:T^*M \to M$. Elements in $T^*M$ are of the form  $(q,p)$ where $q \in M$ and  $p \in T^*_qM$. The cotangent bundle carries the the Liouville 1-form $\vartheta$ (or tautological one-form)  and the symplectic structure $\omega=  d\vartheta$. In local coordinates $(q,p)\in T^*M$, we have 
\[
\vartheta_{\textup{local}}= p\, dq \quad \textrm{ and }\quad \omega= dp \wedge dq. 
\]
The vertical space 
\[
T^v_zT^*M= \ker D\tau^*(z)\cong T^*_qM,\quad z=(q,p) \in T^*M 
\]
is  a Lagrangian subspace of $(T_zT^*M, \omega_z)$. 

A $T$-periodic Hamiltonian, i.e. a smooth function  $H: \T \times T^*M\to \R$ determines a $T$-periodic vector field, the Hamiltonian vector field $X_H$  defined by 
\begin{equation}
	\omega\big(X_H(t,z), \zeta\big)= -DH(t,z)[\zeta], \qquad \forall\, \zeta \in T_zT^*M
\end{equation}
whose corresponding Hamiltonian equation is given by 
\begin{equation}\label{eq:HS-manifold-NA}
	z'(t)= X_H\big(t, z(t)\big), \qquad t \in (0,T)
\end{equation}
which in local coordinates can be written as 
\begin{equation}\label{eq:HS-manifold-NA-local-coord}
\begin{cases}
		q'(t)= \partial_p H(t,q,p), & t \in (0,T)\\
		p'(t)=-\partial_q H(t,q,p)
\end{cases}
\end{equation}
It is well-known that  under the assumption (N1) the Legendre transform 
\begin{equation}
\mathscr L_L:\T\times TM \to \T\times T^*M, \qquad (t,q,v)\mapsto\Big(t,q,DL(t,q,v)\big\vert_{T^v_{(q,v)}TM}\Big)
\end{equation}
is a local smooth diffeomorphism.  The Fenchel transform of $L$ is the non-autonomous Hamiltonian on $T^*M$ 
\begin{equation}
H(t,q,p)=\max_{v \in T_qM}\big(p[v]-L(t,q,v)\big)=p[v(t,q,p)]- L(t,q,v(t,q,p)),	
\end{equation}
for every $(t,q,p) \in \T \times T^*M$, where the map $v$ is a component of the fiber-preserving diffeomorphism 
\[
\mathscr L_L^{-1}: \T \times T^*M \to \T \times TM, \qquad (t,q,p)\mapsto (t,q,v(t,q,p))
\]
the inverse of $\mathscr L_L$.
The $T$-periodic Hamiltonian $H$ on $T^*M$ and the canonical symplectic form $\omega$ on $T^*M$ define a time-dependent  Hamiltonian vector field $X_H$ on $T^*M$, and $Y_L$ previously defined is nothing but the pull-back of $X_H$ by the diffeomorphism $\mathscr L_L$. Otherwise stated, $Y_L$ is the time-dependent $T$-periodic Hamiltonian vector field on $TM$ determined by the symplectic form $\mathscr L_L^*\omega$ and by the Hamiltonian $H \circ \mathscr L_L$.
\begin{note}\label{not:periodic}
We denote by
\begin{itemize}
\item $\ptl$   the set of  $T$-periodic solutions of the Euler-Lagrange Equation~\eqref{eq:EL}
\item $\pth$ for denoting the set of $T$-periodic solutions of the Hamiltonian Equation~\eqref{eq:HS-manifold-NA}
\end{itemize}
If $x \in \ptl$ we let $z_x \in \pth$ defined  by 
\begin{equation}\label{eq:x-vs-z}
\big(t, z_x(t)\big)= \mathscr L_L\big(t,x(t), x'(t)\big)
\end{equation}
the corresponding $T$-periodic solution of the Hamiltonian equation on $T^*M$. Finally,  $\phi_H^t$ denotes the integral flow of the Hamiltonian vector field $X_H$. 
\end{note}
We denote by $T^hT^*M$ (resp. $T^v T^*M$) the horizontal (resp. vertical) tangent bundle and we observe that 
\[
z^*_x(TT^*M)= x^*(T^hT^*M)\oplus x^*(T^vT^*M)\cong x^*(TM)\oplus x^*(T^*M).
\]
By linearizing the  Hamiltonian system given in Equation~\eqref{eq:HS-manifold-NA} along $z_x \in \pth$ defined by Equation~\eqref{eq:x-vs-z} in the direction $ \zeta \in z_x^*(TT^*M)$, we get 
\begin{equation}\label{eq:HS-manifold-NA-lin}
	 \zeta'(t)= JB(t)\zeta(t),  \qquad \forall\, t \in (0,T)  
\end{equation}
where $B(t)\=D^2H\big(t, z_x(t)\big)$. 


\subsection{Variational setting and spectral index}

Let $\Lambda^1(M)$ be the Hilbert manifold of all loops $x: \T \to M$ with Sobolev class $H^1$. The tangent space  $T_x\Lambda^1(M)$ to $\Lambda^1(M)$  at $x$ can be identified in a natural way with the Hilbert space of $T$-periodic $H^1$ (tangent) vector fields along $x$, i.e. 
\begin{equation}\label{eq:hilbert-space}
\mathcal H(x)=\Set{\xi \in H^1(\T, TM)| \tau \circ \xi=x}.
\end{equation}
The non-autonomous Lagrangian function $L$ defines the  {\em Lagrangian (action) functional\/} $\mathbb E: \Lambda^1(M) \to \R$ as follows
\begin{equation}
\mathbb E(x)\= \int_0^T L\big(t,x(t),x'(t)\big) \, d t.
\end{equation}
It is worth noticing that  under the  (N2) assumption, the Lagrangian functional  is of regularity class $\mathscr C^2$ (actually it is smooth). Let $\langle\langle \cdot, \cdot \rangle \rangle_1$ denote the $H^1$-Riemannian metric on $\Lambda^1(M)$ defined by 
\begin{equation}\label{eq:productonloopspace}
\langle\langle \xi, \eta\rangle \rangle_1\= \int_0^T	 \big[\rie{\Ddt\xi}{\Ddt \eta} +\rie{\xi}{\eta}\big]\, dt,  \qquad  \forall\, \xi, \eta \in\mathcal H(x).
\end{equation}
The first variation of $\mathbb E$ at $x$ is given by 
\begin{multline}\label{eq:first-variation}
d\mathbb E(x)[\xi]= \int_0^T\big[D_q L\big(t,x(t), x'(t)\big)[\xi] + D_v L\big(t,x(t), x'(t)\big)[\Ddt\xi]\big]\, dt	\\= 
\int_0^T \Big[\langle \nabla_q L\big(t,x(t), x'(t)\big), \xi \rangle_g  +\langle  \nabla_v L\big(t,x(t), x'(t)\big),\Ddt\xi\rangle_g \Big]\, dt \\
=\int_0^T \left[\left\langle \nabla_q L\big(t,x(t), x'(t)\big)- \Ddt\nabla_v L\big(t,x(t), x'(t)\big),\xi\right \rangle_g \right]\, dt \\+ 
\langle  \nabla_v L\big(t,x(t), x'(t)\big),\xi\rangle_g\Big\vert_{t=0}^T.
\end{multline}
By Equation~\eqref{eq:first-variation} and up to standard regularity arguments we get that critical points  of $\mathbb E$ are $T$-periodic solutions of Equation~\eqref{eq:EL}. Now, being $\mathbb E$ smooth it follows that the first variation  $d\mathbb E(x)$  at $x$ coincides with the Fréchét differential $D\mathbb E(x)$ and   $D^2 \mathbb E(x)$ coincides with the second variation of $\mathbb E$ at $x$ given by 
\begin{multline}\label{eq:second-variation}
	d^2 \mathbb E(x)[\xi,\eta]= \int_0^T \Big[D_{vv} L\big(t, x(t), x'(t)\big)[\Ddt\xi, \Ddt\eta]+ D_{qv} L\big(t, x(t), x'(t)\big)[\xi, \Ddt\eta]\\+ D_{vq} L\big(t, x(t), x'(t)\big)[\Ddt\xi, \eta]+ D_{qq} L\big(t, x(t), x'(t)\big)[\xi, \eta]\Big] dt\\= 
	\int_0^T \Big[ \langle \nabla_{vv} L\big(t, x(t), x'(t)\big)\Ddt\xi, \Ddt\eta \rangle_g +  \langle \nabla_{qv} L\big(t, x(t), x'(t)\big)\xi, \Ddt\eta\rangle_g\Big.\\	+\Big.\langle \nabla_{vq} L\big(t, x(t), x'(t)\big)\Ddt\xi, \eta\rangle_g+ \langle \nabla_{qq} L\big(t, x(t), x'(t)\big)\xi, \eta\rangle_g\Big]\, dt
\end{multline}
We set
\begin{multline}
	\bar P(t)\= \nabla_{vv} L\big(t, x(t), x'(t)\big),  \qquad \bar Q(t)\= \nabla_{qv} L\big(t, x(t), x'(t)\big)\\
	\bar R(t)\= \nabla_{qq} L\big(t, x(t), x'(t)\big),  \quad \trasp{\bar{Q}}(t)\= \nabla_{vq} L\big(t, x(t), x'(t)\big)
\end{multline}
Under the above notation and integrating by parts in the second variation, we get  
\begin{multline}\label{eq:index-form-on-manifold-2} 
D^2 \mathbb E_x[\xi,\eta]=\int_0^T\langle -\Ddt\big[\bar{P}(t)\nabla_t \xi+ \bar{Q}(t)\xi\big]+ \trasp{\bar{Q}}(t)\Ddt\xi+\bar{R}(t)\xi,\eta\rangle_g dt\\+\big[\langle \bar{P}(t)\Ddt \xi+ \bar{Q}(t)\xi,\eta\rangle_g\big]^T_{t=0}.
 \end{multline}
 By Equation  \eqref{eq:index-form-on-manifold-2} it follows  that $ \xi \in \ker D^2 \mathbb E_x$  if and only if $\xi$ is a $H^2$ vector field along $x$ which solves weakly (in the Sobolev sense)  the following boundary value problem
\begin{equation}\label{eq:Sturm-bvp-manifolds}
\begin{cases}
	-\Ddt\big(\bar{P}(t)\nabla_t \xi+ \bar{Q}(t)\xi\big)+ \trasp{\bar{Q}}(t)\Ddt\xi+\bar{R}(t)\xi=0, \qquad t \in (0,T)\\
	\\
	\xi(0)=\xi(T), \qquad \nabla_t \xi(0)= \nabla_t\xi(T).
\end{cases}	
\end{equation}
By standard bootstrap arguments, it follows that $\xi$ is also a classical (smooth) solution of Equation~\eqref{eq:Sturm-bvp-manifolds}. Borrowing the notation from \cite{MPP05, MPP07} we shall refer to such a $\xi$ as {\em $T$-periodic perturbed Jacobi field.\/}


\section{Spectral index, geometrical index and Poincaré map}\label{subsec:ortho-triv}

The goal of this section is to associate at each $T$-periodic solution $x$  of Equation~\eqref{eq:EL} the  {\em spectral index\/} (defined in terms of the spectral flow of path of Fredholm quadratic forms) and to its Hamiltonian counterpart $z_x$ defined in Equation~\eqref{eq:x-vs-z} the {\em geometrical index\/} (defined in terms of a Maslov-type index). We finally introduce the linearized Poincaré map. For the  basic definition and properties of the spectral flow and the Maslov index, we refer to Section \ref{sec:spectral-flow} and Section \ref{sec:Maslov} respectively and references therein.


\subsection{Spectral index: an intrinsic (coordinate free) definition}

Given $x$ be a critical point of $\mathbb E$,  we define for any  $s \in [0, +\infty)$ the bilinear form  $\mathcal I_s: \mathcal H(x) \times \mathcal H(x) \to \R$ given by 
\begin{multline}\label{eq:path-index}
 \mathcal I_s[\xi,\eta]\=d^2 \mathbb E(x)[\xi,\eta]+ s \alpha(x)[\xi,\eta]\\= 
	\int_0^T \Big[ \langle \bar P(t)\Ddt\xi, \Ddt\eta \rangle_g +  \langle \bar Q(t)\xi, \Ddt\eta\rangle_g+\langle \trasp{\bar{Q}}(t)\Ddt\xi, \eta\rangle_g+ \langle \bar R(t)\xi, \eta\rangle_g\Big]\, dt + s \alpha(x)[\xi,\eta]\\
  \textrm{ where } \alpha(x)[\xi,\eta]\=\int_0^T  \langle \bar P(t)\Ddt\xi, \Ddt\eta \rangle_g\, dt.
\end{multline}
 \begin{note}
In short-hand notation and if no confusion can arise, we set $ \mathcal Q^h\= D^2 \mathbb E(x)$.
\end{note}
\begin{prop}\label{thm:famiglia-Fredholm}
For any $s \in [0, +\infty)$ let $\mathcal Q_s$ denote the quadratic form associated to  $\mathcal I_s$ defined in Equation~\eqref{eq:path-index}. Then 
\begin{enumerate}
	\item[] $s\mapsto \mathcal Q_s$ is a smooth  path of Fredholm quadratic forms onto $\mathcal H(x)$.
	In particular $ \mathcal Q^h$ is a Fredholm quadratic form onto $\mathcal H(x)$.
\end{enumerate}
\end{prop}
\begin{proof}
For each $s \in [0,+\infty)$, we start to consider the bilinear form on $\mathcal H(x)$ defined by 
\begin{equation}\label{eq:alpha}
\beta_s(x)[\xi,\eta]\=(1+s) \int_0^T \big[ \langle \bar P(t) \Ddt\xi, \Ddt\eta \rangle_g + 
\langle   \bar P(t) \xi, \eta \rangle_g\big]\, dt =(1+s) \langle\langle \bar {\mathcal P}(x)\xi, \eta\rangle\rangle_1
\end{equation}
where $\bar {\mathcal P}$ is the linear operator on $\mathcal H(x)$ pointwise defined by $\bar P$. By assumption (N1) it follows that $\bar {\mathcal P}$ is invertible on $\mathcal H(x)$ and by assumption (N2), it follows that $\bar {\mathcal P}$ is bounded. Thus $\bar {\mathcal P} \in \GL\big(\mathcal H(x)\big)$ and hence the quadratic form $\mathcal Q_s^\beta(x)$ associated to $\beta$  is a Fredholm quadratic form. We let
\begin{equation}
	\gamma_s(x)[\xi,\eta]\= \int_0^T \Big[\langle \bar{Q}(t)\xi,\nabla_t \eta\rangle_g+\langle \trasp{\bar{Q}}(t)\nabla_t \xi,\eta\rangle_g+\langle \big[\bar{R}(t)-s \bar{P}(t)\big]\xi,\eta\rangle_g \Big]\, dt. 
\end{equation}
The quadratic form $\mathcal Q_s^\gamma(x)$ associated to $\gamma(x)$ is the restriction to $\Lambda^1(M)$ of a quadratic form defined on $\mathscr C^0(x)$ of all continuous vector fields along $x$. Since the embedding  $ \mathcal H(x) \hookrightarrow \mathscr C^0(x)$ is a compact operator (by the Sobolev embedding theorem), it follows that $\mathcal Q_s^\gamma(x)$ is weakly  continuous. 

Now, since  $\mathcal Q_s(x)\= [\mathcal Q_s^\beta+ \mathcal Q_s^\gamma](x)$ the conclusion readily follows by observing that $\mathcal Q_s^h(x)$ is a weakly continuous perturbation of a Fredholm quadratic form. 
This concludes the proof.
\end{proof}
 Taking into account Proposition \ref{thm:famiglia-Fredholm}, we are entitled to give the following definition. 
\begin{defn}\label{def:spectral-index}
	Let $x \in \ptl$. We term {\em spectral index  of $x$\/} the integer 
	\begin{equation}\label{eq:spectral-index-manifold}
	\ispec(x)\= \spfl\big(\mathcal Q_s, s \in [0, s_0]\big)
	\end{equation}
where the (RHS) denotes the spectral flow of the path of Fredholm quadratic forms for a large enough $s_0 >0$. (Cfr. Section \ref{sec:spectral-flow} and references therein for the definitions and its basic properties about the spectral flow).
\end{defn}
\begin{rem}
It is worth noticing that Definition \ref{def:spectral-index} is well-posed in the sense that it doesn't depend upon $s_0$. This fact will be proved in the sequel and it is actually a direct consequence of Lemma \ref{thm:non-degenerate-s_0-forme}. 
\end{rem}
\begin{prop}\label{thm:spectral-index-morse-index}
	If $L$ is $\mathscr C^2$-strictly convex on the fibers of $TM$, meaning that there exists $\ell_0>0$ such that 
	\[
	\nabla_{vv}L\big(t,q,v\big) \geq \ell_0 \Id
	\]
	for every $(t,q,v)\in \T \times TM$, then 
	the Morse index of $x$ (i.e. the dimension of the maximal negative subspace of the Hessian of $\mathcal Q^h$) is finite and 
	\[
	\ispec(x)= \iMorse(x).
	\]
\end{prop}
\begin{proof}
This result follows by observing that if $L$ is $\mathscr C^2$-strictly convex on the fibers of $TM$, then $\mathcal Q^h$ is a  positive Fredholm quadratic form and hence $\mathcal Q_s$ is a path of essentially positive Fredholm quadratic forms (being a weakly compact perturbation of a positive definite quadratic form). In particular the  Morse index of $\mathcal Q_s$ is  finite  for every $s \in [0,+\infty)$. If $s_0$ is large enough the form $\mathcal Q_{s_0}$ is non-degenerate and begin also semi-positive is actually positive definite. Thus its Morse index vanishes. Now, since for path of essentially positive Fredholm quadratic it holds that 
\[
\spfl(\mathcal Q_s, s\in [0,s_0])=\iMorse (\mathcal Q_0)- \iMorse (\mathcal Q_{s_0})= \iMorse (\mathcal Q_0)
\]
the conclusion follows. 
\end{proof}


\subsection{Pull-back bundles and push-forward of Fredholm forms}

We denote by $\mathcal E$ the $\rie{\cdot}{\cdot}$-orthonormal and parallel frame $\EE$, pointwise given by 
\[
\mathcal E(t)=\{e_1(t), \dots, e_n(t)\} 
\]
and, if $x \in \ptl$,  we let $\bar A: T_{x(0)}M\to T_{x(T)}M\cong T_{x(0)}M$ the $\rie{\cdot}{\cdot}$-ortogonal operator defined by 
\[
\bar A e_j(0)= e_j(T).
\] 
 Such a frame $\EE$, induces a trivialization of the pull-back bundle $x^*(TM)$ over $[0,T]$ through the smooth curve $x:[0,T] \to M$; namely  the  smooth one parameter family  of isomorphisms  
\begin{multline}\label{eq:parallel-frame}
[0,T] \ni t \longmapsto E_t \quad \textrm{ where } \quad E_t: \R^n \ni e_i \longmapsto e_i(t) \in  T_{x(t)}M\qquad \forall\ t \in [0,T] \textrm{ and } i =1, \ldots, n\\
\textrm{  are such that } \langle E_t e_i, E_t e_j\rangle_g=  \delta_{ij} \textrm { and } \nabla_t E_t e_i=0
\end{multline}
here $\{e_i\}_{i=1}^n$ is the canonical basis of $\R^n$ and  $\delta_{ij}$ denotes the Kronecker symbol.  

By Equation~\eqref{eq:parallel-frame} we get that the pull-back by $E_t$ of the metric $\rie{\cdot}{\cdot}$ induces  the Euclidean product on $\R^n$ and moreover this pull-back is independent on $t$, as directly follows by the orthogonality assumption on the frame $\EE$.

Let $x$ be a $T$-periodic solution of Equation~\eqref{eq:EL}  and let $z_x$ be defined by Equation \eqref{eq:x-vs-z}. The orthonormal trivialization given in Subsection \ref{subsec:ortho-triv} leads to a  unitary trivialization $\Phi^\EE$ of $z_x(TT^*M)$ over $[0,T]$, i.e. to the smooth one parameter family of isomorphisms 
\begin{multline}
[0,T] \ni t \longmapsto \Phi^\EE_t \quad \textrm{ where } \quad \Phi^\EE_t:  \R^n \oplus \R^n \to x^*(TM)\oplus  x^*(T^*M)\cong z_x(TT^*M) \\ \textrm{ is defined by } \qquad  \Phi^\EE_t\= \begin{bmatrix} 
E_t& 0 \\ 0 & E_t 
\end{bmatrix}.
\end{multline}
We set $A\=E_0^{-1}\bar A^{-1} E_T \in \OO(n)$ and we define
\begin{equation}\label{eq:Ad}
	A_d\= \begin{bmatrix}
	A & 0 \\ 0 & A
 \end{bmatrix}
\end{equation} 
\begin{rem}
We observe that the  $\EE$ (resp.  $\Phi^\EE$)  is, in general, a trivialization of the pull-back bundle $x^*(TM)$ (resp. $z_x(TT^*M)$) over $[0,T]$ (only!) and not on $\T$. However, if $x:\T\to M$  is an orientation preserving smooth curve, then the trivialization $\EE$ (resp. $\Phi^\EE$) can be chosen periodic, namely $E_0=E_T$ (resp. $\Phi^\EE_0= \Phi^\EE_T$). In this particular case $A$ (resp.  $A_d$) reduces to the identity matrix.  In the not orientation preserving case, $A$ (resp. $A_d$)  are always different from $\Id$ and no periodic orthonormal trivializations of $x^*(TM)$ (resp. symplectic trivialization of  $z_x^*(TT^*M)$) can be chosen. 
\end{rem}

Let us now consider the Hilbert space 
\begin{equation}\label{eq:periodic-vf}
	H^1_A([0,T], \R^n)=\Set{u \in H^1([0,T], \R^n)| u(0)=A u(T)}
\end{equation}
equipped with the inner product 
\begin{equation}\label{eq:3-5pigeon}
	\langle \langle v,w\rangle\rangle _A\= \int_0^t \big[\langle v'(s), A w'(s)\rangle+ \langle v(s), A w(s)\rangle\big]\, ds.
\end{equation}
Denoting by   $\Psi: \mathcal H(x) \to H^1_A([0,T], \R^n)$  the map defined by $\Psi(\xi)=u$ where $u(t)=E_t^{-1}(\xi(t))$, it follows that $\Psi$ is a linear isomorphism and it is easy to check that 
 \begin{multline}
\xi(0)=	\xi(T)  \quad \iff\quad E_0 u(0)= E_T u(T) \quad\iff\quad u(0)= A u(T) \textrm{ and } \\
\nabla_t \xi(0)= \nabla_t \xi(T)  \quad \iff\quad u'(0)= A u'(T)
\end{multline}
where in the last we used the parallel property of the frame.

For $i=1, \dots, n$ and $t \in[0,T]$, we let $e_i(t)\=E_t e_i$ and we denote by $\langle P(t) \cdot, \cdot \rangle$, $\langle Q(t) \cdot, \cdot \rangle$ and $\langle R(t) \cdot, \cdot \rangle$ respectively the pull-back by $E_t$ of $\rie{\bar  P(t) \cdot}{ \cdot }$, $\rie{ \bar Q(t) \cdot}{\cdot }$ and $\rie{ \bar R(t) \cdot}{\cdot }$. Thus, we get
\begin{multline}
P(t)\=[p_{ij}(t)]_{i,j=0}^n, \quad Q(t)\=[q^{ij}(t)]_{i,j=0}^n,\quad  R(t)\=[r_{ij}(t)]_{i,j=0}^n\quad  \textrm{ where }\\
p_{ij}(t)\=\rie{\bar{P}(t)e_i(t)}{e_j(t)}, \quad q_{ij}(t)\=\rie{\bar{Q}(t)e_i(t)}{e_j(t)},  \quad  r_{ij}(t)\	=\rie{ \bar{R}(t)e_i(t)}{e_j(t)}.
\end{multline}
We observe that $P$ and $ R$ are symmetric matrices and being  $e_i(T)=\sum_{j=1}^na_{ij}e_j(0)$ we get also that 
\begin{equation}\label{eq:condition-of-P-after-trivialization}
P(0)=AP(T)\trasp{A},\qquad P'(0)=AP'(T)\trasp{A}, \qquad Q(0)=AQ(T), \qquad R(0)=AR(T)\trasp{A}.
\end{equation}
Now, for every $s \in [0,+\infty)$, the push-forward by $\Psi$ of the index forms $\mathcal I_s$ on $\mathcal H(x)$ is given by the  bounded symmetric bilinear forms  on $H^1_A([0,T], \R^n)$ defined by 
\begin{multline}\label{eq:path-index-2}
 I_s[u,v]= 
	\int_0^T \Big[ \langle P(t)u'(t), v'(t) \rangle +  \langle  Q(t)u(t), v'(t)\rangle +\langle \trasp{Q}(t)u'(t), v(t)\rangle+ \langle  R(t)u(t), v(t)\rangle\Big]\, dt \\+ s \alpha(x)[u,v]\quad 
  \textrm{ where } \quad \alpha(x)[u,v]=\int_0^T  \langle  P(t)u'(t), v'(t) \rangle\, dt.
\end{multline}
Denoting by $q^A_s$ the quadratic form on $H^1_A([0,T], \R^n)$ associate to $I_s$ then, as direct consequence of Proposition  \ref{thm:famiglia-Fredholm}, we get also the following result.  
\begin{lem}\label{thm:riduzione-Hilbert}
For every $s \in [0,+\infty)$, the quadratic form $q_s$ is Fredholm on $H^1_A([0,T], \R^n)$. 
\end{lem}
\begin{proof}
We observe that, for each $s \in [0,+\infty)$ the form $I_s$ is the push-forward by the linear isomorphism $\Psi$ of $\mathcal I_s$. Now the conclusion of this result is a consequence of   Proposition \ref{thm:famiglia-Fredholm} and the fact that the spectral flow for  a  generalized family of Fredholm quadratic forms doesn't depend on the chosen trivialization. (Cfr. Section \ref{sec:spectral-flow} for further details). This concludes the proof. 
\end{proof}
The following result is crucial in the well-posedness of the spectral index. 
\begin{lem}\label{thm:non-degenerate-s_0-forme}
Under the above notation, there exists $s_0 \in [0,+\infty)$ large enough such that for every $s \geq s_0$, the form 
  $I_s$ given in Equation \eqref{eq:path-index-2} is non-degenerate (in the sense of bilinear forms).
\end{lem}
\begin{proof}
We argue by contradiction and we assume that for every $s_0 \geq 0$ there exists $s \geq s_0$ such that  $I_s$ is degenerate. 
Thus we get
\begin{multline}\label{eq:degenerate condition for index form}
0\equiv I_s(x,y)=\int_{ 0 }^{ T }\left\langle -\dfrac{d}{dt}\big(P(t) x'(t)+ Q(t)x(t)\big)+ \trasp{Q}(t)x'(t)+R(t)x(t)+sP(t)x(t),y(t)\right\rangle dt\\
+\langle P(t) x'(t)+ Q(t)x(t),y(t)\rangle|_{t=0}^T\\
=\int_{ 0 }^{ T }\langle P(t) x'(t), y'(t)\rangle+\langle Q(t)x(t), y'(t)\rangle+\langle \trasp{Q}(t) x'(t),y(t)\rangle+\langle R(t)x(t),y(t)\rangle +\langle sP(t)x(t),y(t)\rangle dt.
\end{multline}
We let $y(t)\=P(t)x(t)$ and we observe that as direce consequence of  Equation~\eqref{eq:condition-of-P-after-trivialization} it is admissible (meaning that $y$ belongs to $H^1_A$). Now, by replacing $y$ into Equation~\eqref{eq:degenerate condition for index form} with $y=Px$,  we get 
\begin{multline}\label{eq:inequality of non-degenerate condition for index form}
0\equiv I(x,Px)
=\int_{ 0 }^{ T }\Big[\langle P(t)x'(t),P(t) x'(t)\rangle+\langle  P(t)x'(t), P'(t)x(t)\rangle+\langle Q(t)x(t), P'(t)x(t)\rangle \Big.\\ \Big.+\langle Q(t)x(t), P(t)x'(t)\rangle+\langle \trasp{Q}(t) x'(t),P(t)x(t)\rangle\Big.\\\Big.+\langle R(t)x(t),P(t)x(t)\rangle  +\langle sP(t)x(t),P(t)x(t)\rangle \Big]dt\\
\geq\int_{ 0 }^{ T }\Big[\Vert P(t)x'(t)\Vert^2-(\Vert P'(t)P^{-1}(t)\Vert+\Vert Q(t)P^{-1}(t)\Vert\Vert P'(t)P^{-1}(t)\Vert\Big.\\
\Big.+\Vert \trasp{Q}(t)P^{-1}(t)\Vert)\Vert P(t)x(t)\Vert \Vert P(t)x'(t) \Vert\\
 +(s-\Vert R(t)P^{-1}(t)\Vert- \Vert Q(t)P^{-1}(t) \Vert-\Vert  P'(t)P^{-1}(t)\Vert)\Vert P(t)x(t)\Vert^2 \Big]\, dt>0
 \end{multline}
for $s$ large enough. By this contradiction we conclude the proof.
\end{proof}
\begin{prop}\label{thm:spectral-index-well-defined}
Let $x \in \ptl$. Then the spectral index  given in Definition \ref{def:spectral-index} is well-defined.
\end{prop}
\begin{proof}
We start to observe that as consequence of Lemma \ref{thm:riduzione-Hilbert},  $s\mapsto q_s$ is a path of Fredholm quadratic forms on $H^1_A([0,T], \R^n)$. Moreover, by Lemma \ref{thm:non-degenerate-s_0-forme}, there exists $s_0 \in [0,+\infty)$ such that $q_s$ is non-degenerate for every $s \geq s_0$ and hence the integer $\spfl(q_s, s \in [0,s_0])$ is well-defined.  

The conclusion follows by observing that $q_s$ is the push-forward by $\Psi$  of the Fredholm quadratic form  $\mathcal Q_s$ and by the fact that the spectral flow of a generalized family of Fredholm quadratic forms on the (trivial) Hilbert bundle $[0,s_0] \times \mathcal H(x)$ is independent on the trivialization. This concludes the proof. 	
	\end{proof}
The next step is to prove that the spectral index actually can be reduced to the spectral flow of a path of second order self-adjoint Sturm-Lioville operators. 

First of all we observe that the boundary value problem appearing  in Equation~\eqref{eq:path-index}, reduces to the Sturm-Liouville boundary value problem in $\R^n$ given by 
\begin{equation}\label{eq:Sturm-Liouville equation after trivialization}
\begin{cases}
	-\dfrac{d}{dt}\left(P(t)\dfrac{d}{dt}u(t)+Q(t)u(t)\right)+\trasp{Q}(t)\dfrac{d}{dt}u(t)+\big(R(t)+sP(t)\big) u(t)=0, \qquad t \in (0,T)\\
 \\
	u(0)=A u(T), \qquad u'(0)= Au'(T).
\end{cases}	
\end{equation}
Now, for any $s \in [0,+\infty)$,  let $ s \mapsto \mathcal A_s$ be  the  one parameter family of operators pointwise defined by 
\begin{equation}\label{eq:definition of second order operators}
\mathcal A_s: \mathcal D\subset L^2([0,T], \R^n) \to L^2([0,T], \R^n),  \quad  
\mathcal A_s\=-\dfrac{d}{dt}\left(P(t)\dfrac{d}{dt}+Q(t)\right)+\trasp{Q}\dfrac{d}{dt} +R(t)+sP(t)
\end{equation}
where $\mathcal D\=\Set{u \in H^2([0,T], \R^n)| u(0)=Au(T), u'(0)= Au' (T)}$. It is worth to observe that $\mathcal A_s$ is an unbounded  self-adjoint (in $L^2$) Fredholm operator with dense domain $\mathcal D$.

\begin{cor}\label{thm:non-degenerate-s_0-operatori}
Let  $s \in [0, +\infty)$ and let $\mathcal{A}_s$ denote the Sturm-Liouville operator defined  in Equation~\eqref{eq:definition of second order operators}.  Then, there exists $s_0 \geq 0$ large enough such that  $\mathcal{A}_s$ is  non-degenerate, namely $\ker \mathcal A_s =\{0\}$  for every $s\geq s_0$.
\end{cor}
\begin{proof}
We argue by contradiction and we assume that for every $s_0 \geq 0$ there exists $s \geq s_0$ such that  $	ker \mathcal{A}_s \neq 0$. Thus,  there exists $x\in \mathcal D$ such that $\mathcal{A}_s x=0$. Thus, for every $y \in \mathcal D$, we get 
\begin{multline}\label{eq:degenerate condition for index form}
0\equiv I_s(x,y)=\int_{ 0 }^{ T }\left\langle -\dfrac{d}{dt}\big(P(t) x'(t)+ Q(t)x(t)\big)+ \trasp{Q}(t)x'(t)+R(t)x(t)+sP(t)x(t),y(t)\right\rangle dt\\
+\langle P(t) x'(t)+ Q(t)x(t),y(t)\rangle|_{t=0}^T\\
=\int_{ 0 }^{ T }\Big[\langle P(t) x'(t), y'(t)\rangle+\langle Q(t)x(t), y'(t)\rangle+\langle \trasp{Q}(t) x'(t),y(t)\rangle\\+\langle R(t)x(t),y(t)\rangle +\langle sP(t)x(t),y(t)\rangle\Big]\, dt.
\end{multline}
The conclusion readily follows by Corollary \ref{thm:non-degenerate-s_0-forme}. 
\end{proof}

\begin{prop}\label{thm:definition of spectral index}
Let $x\in \ptl$ and $s_0$ be given in Proposition \ref{thm:famiglia-Fredholm}.   Then the following equality holds
 \begin{equation}\label{eq:definition of spec index}
\ispec(x)=\spfl(\mathcal{A}_s, s\in[0,s_0]).
\end{equation}
\end{prop}
\begin{proof}
For the proof of this result we refer the interested reader to \cite[Proposition 4.12]{BHPT19}. 
\end{proof}


\subsection{Geometrical index and linearized Poincaré map}\label{subsec:geo-index}

The aim of this paragraph is to associate to each $T$-periodic solution $x \in \ptl$ an integer defined in terms of the  intersection index $\iCLM$. (Cfr. Appendix \ref{sec:Maslov} for the definition and basic properties). 

By trivializing the pull-back bundle $x^*(TM)$ over $TM$ through the frame $\EE$ defined in Subsection \ref{subsec:ortho-triv}, the Jacobi deviation equation along $x$ reduces to the Sturm-Liouville system as given in  Equation~\eqref{eq:Sturm-Liouville equation after trivialization}. By setting   $y(t)=P(t) u'(t)+Q(t)u(t)$ and $z(t)=\trasp{(y(t),u(t))}$ we finally get 
\begin{multline}\label{eq:Hamilton system}
\begin{cases}
z'(t)=JB(t)z(t), \qquad t \in [0,T]\\
z(0)=A_d z(T)
\end{cases}\quad 
\textrm{ where } \\
B(t)\=\begin{bmatrix} P^{-1}(t) &- P^{-1}(t)Q(t)\\-Q(t)P^{-1}(t)& \trasp{Q}(t)P^{-1}(t)Q(t)-R(t) \end{bmatrix} 
\end{multline} 
and $A_d$ has been defined in Equation~\eqref{eq:Ad}.
\begin{rem}
We would like to stress the reader on the notational aspect introduced above in this paragraph. More precisely, on the fact that, for computational convenience and for uniformity with the literature on the subject, in the first component of $z$ is placed the momentum coordinate and in the second the position.     
\end{rem}
In the standard symplectic space $(\R^{2n}, \omega)$, we denote by $J$ the standard symplectic matrix defined by $J=\begin{bmatrix} 0&-\Id\\ \Id &0\end{bmatrix}$. Thus the symplectic form $\omega$ can be  represented with respect to the Euclidean product $\langle\cdot, \cdot\rangle$ by $J$ as follows $\omega(z_1,z_2)=\langle J z_1,z_2\rangle$ for every $z_1, z_2 \in \R^{2n}$. 

Now, given $M \in \Sp(2n, \R)$, we denote by $\Gr(M)=\{(x,Mx)|x\in \R^{2n}\}$ its graph and we recall that  $\Gr(M)$ is a Lagrangian subspace of the symplectic  space $(\R^{2n} \times \R^{2n}, -\omega \times \omega)$.  
\begin{defn}\label{def:geometrical-index-NA}
Let $x$ be a $T$-periodic solution of Equation~\eqref{eq:EL},  $z_x$ be defined in Equation \eqref{eq:x-vs-z} and let us consider the path 
\begin{equation}
	\gamma_\Phi :[0,T] \to \Sp(2n, \R) \quad \textrm{ given  by }\quad  \gamma_\Phi(t)\= A_d[\Phi^E(t)]^{-1} D \phi_H^t(z_x(0))\Phi^E(0).
\end{equation}
We define the {\em geometrical index of $x$\/} as follows
\begin{equation}\label{eq:geo-na}
\igeo(x)\=\iCLM(\Delta,\Gr(\gamma_\Phi(t)), t\in[0,T])
\end{equation}
where the (RHS) in Equation~\eqref{eq:definition of maslov index}  denotes  the $\iCLM$ intersection index between the Lagrangian path $t \mapsto \Gr(\gamma_\Phi(t))$ and the Lagrangian path $\Delta\=\Gr(\Id)$.
\end{defn}
Let $x \in \ptl$ and $z_x$ be given in Equation \eqref{eq:x-vs-z}.
We can define the {\em linearized Poincaré map of $z_x$\/} as follows. 
\begin{multline}\label{eq:linearized-poincare-map}
\mathfrak P_{z_x}:  T_{x(0)}M\oplus T_{x(0)}^*M \to 
T_{x(0)}M\oplus T_{x(0)}^*M  \textrm{ is   given by 	}\\
\mathfrak P_{z_x}(\alpha_0, \delta_0)\=\bar A_d\trasp{\Big( \zeta(T),\bar P(T) \Ddt \zeta(T)+ \bar Q(T)\zeta(T)\Big)}\\ \quad \textrm{ for } \quad 
	\bar A_d \= \begin{bmatrix}
		\bar A & 0 \\ 0 & \bar A
	\end{bmatrix}
\end{multline}
where $\zeta$ is the unique vector field along $x$ such that $\zeta(0)=\alpha_0$ and  $\bar P(0) \Ddt \zeta(0)+ \bar Q(0)\zeta(0)=\delta_0$. Fixed points of $\mathfrak P_{z_x}$ corresponds to periodic vector fields along $z_x$. 
\begin{defn}\label{def:spectral-stability-manifold}
	We term $z_x \in \pth$ {\em spectrally stable\/} if the spectrum $\spec({\mathfrak P_{z_x}}) \subset \U$ where $\U\subset \C$ denotes the unit circle of the complex plane. Furthermore, if $\mathfrak P_{z_x}$ is also semi-simple, then $z_x$ is termed {\em linearly stable\/}.
\end{defn}
By pulling back the linearized Poincaré map defined in Equation~\eqref{eq:linearized-poincare-map} through the unitary  trivialization $\Phi^\EE$ of $z_x(TT^*M)$ over $[0,T]$ we get the map 
\begin{equation}\label{eq:linearized-poincare-euclidea}
	P^\EE: \R^n \oplus \R^n \to \R^n \oplus \R^n \textrm{ defined by } 
	P^\EE(y_0, u_0)=A_d\trasp{\big(Pu'(T)+ Q u(T),  u(T)\big)}
\end{equation}
where $z(t)=\big(y(t),u(t)\big)$ is the unique solution of the Hamiltonian system given in Equation~\eqref{eq:Hamilton system} such that $z(0)=(y_0, u_0)$.

Denoting by  $t \mapsto\psi(t)$ the fundamental solution of (linear) Hamiltonian system  given in Equation~\eqref{eq:Hamilton system}, then we get the geometrical index given in Definition \ref{def:geometrical-index-NA} reduces to  
\begin{equation}\label{eq:definition of maslov index}
\igeo(x)\=\iCLM(\Delta,\Gr(A_d\psi(t)), t\in[0,T]).
\end{equation}
Moreover,  a $T$-periodic solution $x$ of the Equation~\eqref{eq:EL} is  {\em linearly stable\/} if the symplectic matrix $A_d\psi(T)$ is linear stable.

Next result is one of the two main ingredients that we need for proving our main results and in particular relates the parity of the $\iomega{1}$-index to the linear instability of the periodic orbit. 
\begin{lem}\label{lem:instability by maslov index}
Let  $x\in\ptl$. Then the following implication holds
\[ 
\igeo(x) \textrm{ is odd } \Rightarrow x \textrm{ is linearly unstable }
\]
\end{lem}
\begin{proof}
To prove this result it is equivalent to prove that if $x$ is linear stable then $\igeo(x)$ is even.

If $x$ is linear stable by the previous discussion, this is equivalent to the linear stability of   $A_d\psi(T)$.  By invoking Lemma \ref{prop:how to know in which component} , we get that  $e^{-\varepsilon J}A_d\psi(T)\in \Sp(2n,\R)^+$ for some sufficiently small $\varepsilon>0$.  

By direct computations we get that 
\begin{equation}\label{eq:computation of A in positive component}
\begin{aligned}
\det(e^{-\varepsilon J}A_d-\Id)&=\det(A_d-e^{\varepsilon J})=\det\begin{bmatrix}A-\cos\varepsilon \Id&\sin\varepsilon \Id\\-\sin\varepsilon \Id&A-\cos\varepsilon \Id\end{bmatrix}\\
&=(-1)^{n(n+1)}(\sin\varepsilon)^{2n}\det(\Id+\frac{(A-\cos\varepsilon \Id)^2}{(\sin\varepsilon)^2})>0.
\end{aligned}
\end{equation}
Then $e^{-\varepsilon J}A_d\in\Sp(2n,\R)^+$. By \cite[Proposition 3.1]{LZ00} it follows that 
\begin{equation}\label{eq:relation between long and clm maslov indices}
\igeo(x)=\iomega{1}\big(A_d\psi(t), t\in[0,T]\big).
\end{equation}
Being $\psi$ a fundamental solution, in particular $\psi(0)=\Id$ and so by these arguments we get that both endpoints $e^{-\varepsilon J}A_d, e^{-\varepsilon J}A_d\psi(T)$ of the  path $t\mapsto e^{-\varepsilon J}A_d\psi(t)$ belong to  $\Sp(2n,\R)^+$. By using the characterization of the parity of the $\iomega{1}$-index given in  Lemma \ref{lem:parity property}, we get that $\iomega{1}(A_d\psi)$ is even. Consequently, by Equation~\eqref{eq:relation between long and clm maslov indices}, we conclude that   $\igeo(x)$ is even. This concludes the proof.
\end{proof}


\section{Some technical lemmas and  proof of the main result}\label{sec:Instability criterion for non-autonomous Lagrangian system}

In this section we collect the main technical lemmas that we need for proving our  main result.

The first result we start with, insures that the Definition \ref{def:spectral-index} of the spectral index is well-posed, meaning that it doesn't depend upon the choice of $s_0$ (appearing in such a definition) as soon as $s_0$ is sufficiently large. 

The key idea behind the proof of the relation between the spectral index and the geometrical index is homotopy-theoretical in its own. 

We consider the  (unbounded) self-adjoint (in $L^2$) Fredholm operator having (dense) domain $\mathcal D$ 
\begin{equation}\label{eq:two parameter operator path}
\mathcal{A}_{c,s}=-\dfrac{d}{dt}\left(P(t)\dfrac{d}{dt}+c Q(t)\right)+c \trasp{Q}(t)\dfrac{d}{dt}+c R(t)+s P(t) \quad \textrm{ for } \  (c,s) \in [0,1]\times [0,s_0].
\end{equation}
As before, we associate to the second order differential operator $\mathcal A_{c,s}$ the first order (Hamiltonian) differential operator given by 
\begin{equation}\label{eq:two parameter operator path-first-order}
\mathcal{J}_{c,s}\=-J\dt-B_{c,s}(t): \widetilde{\mathcal D} \subset L^2([0,T], \R^{2n})\to L^2([0,T], \R^{2n}), \quad \textrm{ for } \  (c,s) \in [0,1]\times [0,s_0]
\end{equation}
where $\widetilde{\mathcal D}= \Set{z \in W^{1,2}([0,T], \R^{2n})| \big(z(0), z(T)\big) \in \Delta}$ and where 
%
%
\begin{equation}\label{eq:bcs}
B_{c,s}(t)=\begin{bmatrix} 
P^{-1}(t) &- c P^{-1}(t)Q(t)\\-c Q(t)P^{-1}(t)& c^2 \trasp{Q}(t) P^{-1}(t) Q(t) -c R(t)-s P(t) 
\end{bmatrix}. 
\end{equation}
\begin{note}
We shall denote by $\psi_{c,s}$ the fundamental solution of the Hamiltonian system 
\begin{equation}\label{eq:hs-general}
		z'(t)= J B_{c,s}(t) z(t), \qquad t \in [0,T].
\end{equation}
\end{note}
The following result is well-known and provides the relation between the spectral flow of the path $s \mapsto \mathcal A_{c,s}$ and the spectral flow of the path $s\mapsto \mathcal J_{c,s}$. 
\begin{lem}\label{lem:equivalence between first and second order spectral flow}
Under the above notation, the following equality holds
\begin{equation}\label{eq:equivalence between first and second order spectral flow}
\spfl(\mathcal{A}_{c,s}, s\in[0,s_0])=\spfl(\mathcal{J}_{c,s}, s\in[0,s_0]), \qquad \forall \,c\in[0,1].
\end{equation}
\end{lem}
\begin{proof}
 After a standard perturbation argument in order the involved paths have only regular crossings, the proof goes on by proving the following two claims
\begin{enumerate}
\item $t_* \in [0,T]$ is a crossing instant for the path $s \mapsto \mathcal A_{c,s}$ if and only if it is  crossing instant for the path $s \mapsto \mathcal J_{c,s}$
\item the crossing forms at each crossing instant are isomorphic.
\end{enumerate}
Once these has been proved, then the thesis follows by the homotopy invariant property of the spectral flow and by  Equation~\eqref{eq:spectral-flow-crossings}. For further details we refer to   \cite[Proposition 4.12]{BHPT19}.
\end{proof}


\subsection{Sturm-Liouville operators having constant principal symbol}

We start to consider the following new two parameters family of (unbounded) self-adjoint (in $L^2$) Fredholm operators with dense domain $\mathcal D$ 
\begin{equation}\label{eq:two parameter operator path}
\overline{\mathcal{A}}_{c,s}=-\dfrac{d}{dt}\left(P\dfrac{d}{dt}+c Q(t)\right)+c \trasp{Q}(t)\dfrac{d}{dt}+c R(t)+s P, \quad \textrm{ for } \  (c,s) \in [0,1]\times [0,s_0]
\end{equation}
obtained by replacing into Equation~\eqref{eq:two parameter operator path}, the principal symbol $P(t)$ of the operator $\mathcal A_{c,s}$ by $\overline P$. The domain $\mathcal D$ of the operator is, as before,  $\mathcal D\=\Set{u \in W^{2,2}([0,T],\R^n)| u(0)=Au(T), u'(0)= Au' (T)}$. 

For $c=0$ and $s=s_0$, it is easy to observe that $x \in \ker \overline{\mathcal A}_{c,s}$ if and only if it is solution of the following Sturm-Liouville boundary value problem 
\begin{equation}\label{eq:reduced second order system}
\begin{cases}
	-\overline Pu''(t)+s_0\overline Pu(t) =0, &   t \in [0,T] \\
	u(0)=Au(T), \quad u'(0)=Au'(T).
\end{cases}
\end{equation}
Let $\overline \psi_{c,s}(t)$ be the fundamental solution of the Hamiltonian system obtained by replacing the path $t\mapsto P(t)$ by the constant matrix $\overline P$ in the matrix appearing in  Equation~\eqref{eq:bcs}; thus  $\overline \psi_{c,s}$ is the fundamental solution of  
\begin{equation}\label{eq:Hamilton system with sP}
	z'(t)= J \overline B_{c,s}(t) z(t), \qquad  t \in [0,T] 
\end{equation}
where we set 
\[
\overline B_{c,s}(t)=\begin{bmatrix} 
\overline P^{-1} &- c \overline P^{-1} Q(t)\\-c Q(t)\overline P^{-1} & c^2 \trasp{Q}(t) \overline P^{-1}Q(t) -c R(t)-s \overline P
\end{bmatrix}.
\]
Actually in the special case $c=0$ and $s=s_0$, it is possible  to compute the intersection index  
$\iCLM\big(\Delta,\Gr(A_d\overline \psi_{0,s_0}(t)), t\in[0,T]\big)$ 
 with respect to the diagonal of the path  $A_d\overline \psi_{0,s_0}(t))$. This is the content of Lemma  \ref{lem:computation of special index p constant}. A key  step in the proof is based on the reduction to the Sturm-Liouville bvp arising in the analogous problem for closed geodesics on semi-Riemannian manifold. (Cfr. \cite[Proposition 5.6]{HPY19}).
\begin{lem}\label{lem:computation of special index p constant}
Under the above notation, the following equality holds:
\begin{equation}
\iCLM\big(\Delta,\Gr(A_d\overline \psi_{0,s_0}(t)), t\in[0,T]\big)=\dim\ker(A-\Id).
\end{equation}
In particular, if $A=\Id$, we get 
\begin{equation}
\iCLM\big(\Delta,\Gr(\overline \psi_{0,s_0}(t)), t\in[0,T]\big)=n.
\end{equation}
\end{lem}
\begin{proof}
Let $k \in \{0, \dots, n\}$ denotes the number of positive eigenvalues of $P$ and we define the matrix  
\begin{equation}\label{eq:G}
G\=\begin{bmatrix}I_k&0\\0&-I_{n-k}\end{bmatrix}.
\end{equation}
 Since $\overline P$ is symmetric, by diagonalizing it in the orthogonal group, we get that  $\overline P=\traspm{M}GM^{-1}$, for some $M\in \OO(n)$. Thus the  system given in Equation~\eqref{eq:reduced second order system} can be re-written as follows
\begin{equation}
\begin{cases}-GM^{-1}u''(t)+s_0 G M^{-1} u(t)=0, \qquad  t \in [0,T] \\
	u(0)=Au(T), \quad u'(0)=Au'(T).
\end{cases}
\end{equation}

Let $x(t)\=M^{-1}u(t)$; thus   we have
\begin{equation}\label{eq:reduce to geodisic paper}
\begin{cases}
-Gx''(t)+s_0 G x(t)=0, &   t \in [0,T] \\
x(0)=M^{-1}AMx(T), \quad  x'(0)=M^{-1}AMx'(T).
\end{cases}
\end{equation}
We set $\widehat{A}\=M^{-1}AM$ and we observe  that $\trasp{\widehat{A}}G\widehat{A}=G$. Let $\widehat{\psi}_{0,s_0}$ be the fundamental solution of the  Hamiltonian system corresponding to the ODE  given in Equation~\eqref{eq:reduce to geodisic paper}. By invoking  \cite[Proposition 5.6]{HPY19}, we can conclude that
\begin{equation}\label{eq:maslov-index-reduce-geodesic-paper}
\iCLM(\Delta, \Gr(\widehat{A}_d\widehat{\psi}_{0,s_0}(t)), t\in[0,T])=\dim\ker(\widehat{A}-\Id),
\end{equation}
where $\widehat{A}_d=\begin{bmatrix}\traspm{\widehat{A}}&0\\0&\widehat{A}\end{bmatrix}$. 
 
The last step for concluding the proof is to show that  \begin{equation}\label{eq:relation between two maslov index for reduced system}
\iCLM(\Delta, \Gr(A_d\overline\psi_{0,s_0}(t)), t\in[0,T])= \iCLM(\Delta,\Gr(\widehat{A}_d\widehat{\psi}_{0,s_0}(t)), t\in[0,T]).
\end{equation}
First of all we start by observing that  $x$ is a solution of  the bvp given in Equation  \eqref{eq:reduced second order system} if and only if $u$  is a solution of the system given in Equation~\eqref{eq:reduce to geodisic paper}. In particular, we get that the crossing instants of the paths $t\mapsto \Gr(A_d \overline \psi_{0,s_0}(t))$  and $t \mapsto \Gr(A_d\widehat{\psi}_{0,s_0}(t))$ with respect to $\Delta$ both coincide. 

Now, respectively speaking, the crossing forms 
 at each crossing corresponding to the two paths,  are the quadratic forms represented by the block diagonal matrices $B_{0,s_0}=\begin{bmatrix} \bar{P}^{-1} &0\\0&-s_0 \bar{P} \end{bmatrix}$ and $\widehat{B}_{0,s_0}=\begin{bmatrix} G &0\\0&-s_0 G \end{bmatrix}$. By setting  $y(t)=\bar{P}u'(t)$ and $v(t)=G x'(t)$, we get  $v(t)=\trasp{M} y(t)$. 

If  $t_0$ is a crossing instant, then we get that 
\[
z=\begin{bmatrix}y\\u\end{bmatrix}\in \ker(A_d\overline \psi_{0,s_0}(t_0)-\Id) \quad \iff\quad  \widehat{z}=\begin{bmatrix}v\\u\end{bmatrix}=\begin{bmatrix}\trasp{M}y\\M^{-1}u\end{bmatrix}\in \ker(\widehat{A}_d\widehat{\psi}_{0,s_0}(t_0)-\Id).
\]
Furthermore, we have 
\begin{equation}\label{eq:relation between two crossing forms}
\begin{aligned}
\langle\widehat{B}_{0,s_0}\widehat{z},\widehat{z} \rangle&=\left\langle \begin{bmatrix} G &0\\0&-s_0G \end{bmatrix}
\begin{bmatrix}\trasp{M}y\\M^{-1}u\end{bmatrix},\begin{bmatrix}\trasp{M}y\\M^{-1}u\end{bmatrix}\right\rangle=\langle MG\trasp{M}y,y\rangle -s_0\langle \traspm{M}GM^{-1}u,u\rangle\\
&= \langle P^{-1}y,y\rangle -s_0\langle Pu,u\rangle=\left\langle\begin{bmatrix} P^{-1} &0\\0&-s_0P \end{bmatrix}\begin{bmatrix}y\\u\end{bmatrix},\begin{bmatrix}y\\u\end{bmatrix}\right\rangle=\langle B_{0,s_0}z,z  \rangle.
\end{aligned}
\end{equation}
The computation performed in Equation~\eqref{eq:relation between two crossing forms}, in particular, shows that the crossing forms are isomorphic, where the isomorphism is provided by the map 
\[
z \mapsto \widetilde M z \quad \textrm{ where }\quad  \widetilde M\= \begin{bmatrix} \trasp{M} & 0\\ 0 & M^{-1}
 \end{bmatrix}.
\]
By invoking Equation~\eqref{eq:iclm-crossings}, we can conclude the proof of the equality given in Equation  \eqref{eq:relation between two maslov index for reduced system}. 

The conclusion readily follows by Equation~\eqref{eq:maslov-index-reduce-geodesic-paper}, Equation~\eqref{eq:relation between two maslov index for reduced system} and Equation~\eqref{eq:trivial} once observed that 
\begin{equation}\label{eq:trivial}
\dim\ker(\widehat{A}-\Id)=\dim\ker(A-\Id). 
\end{equation} 
The second claim readily follows by observing that if $A=\Id$ the kernel dimension of $A-\Id$ is $n$.
\end{proof}


\subsection{Sturm-Liouville operators having non-constant principal symbol}

In the most general case of  $t$-dependent principal symbol $t \mapsto P(t)$  the result is formally the same, but the proof is much more tricky and delicate. There is a deep reason for such a difference; reason closely related to the Kato's Selection Theorem.   

 Proposition \ref{thm:index-formula-for-general-P-1} is formally identical to the second claim in Lemma \ref{lem:computation of special index p constant}, but in the case of non constant principal symbol Sturm-Liouville operators.  
\begin{prop}\label{thm:index-formula-for-general-P-1}
Under the notation above, the following equality holds
\begin{equation}\label{eq:formula-p-generale-1}
\iCLM(\Delta,\Gr(\psi_{0,s_0}(t)), t \in [0,T])=n,
\end{equation}
\end{prop}
\begin{proof}
Before providing the proof of this result we briefly sketch the main steps.  By using the fixed-endpoints homotopy invariance of the $\iCLM$-index, the main idea is to construct a homotopy between the path $t \mapsto P(t)$ and a  non-degenerate constant and symmetric matrix  and then to reduce this computation to Lemma \ref{lem:computation of special index p constant}.  This will be achieved as follows: 
\begin{enumerate}
\item by reducing (through a homotopy argument) $t \mapsto P(t)$ to a non-degenerate $t$-dependent symmetric matrix  having constant spectrum given by $\{\pm 1\}$;
\item by reducing (through a homotopy argument) the non-degenerate $t$-dependent symmetric matrix  having constant spectrum given by $\{\pm 1\}$ to a constant non-degenerate symmetric matrix $\overline P$.
\end{enumerate}
\underline{Step 1}. Reduction of the problem to the case in which the principal symbol is pointwise given by $P_0(t)$.

For every eigenvalue $\lambda(t)\in \sigma(P(t))$, we denote by $E_{\lambda(t)}$ be the eigenspace corresponding to it. We define the  positive and negative eigenspaces of $P(t)$, as follows 
\begin{equation}
E_+(t)=\bigoplus_{0<\lambda(t)\in \sigma (P(t))}E_{\lambda(t)} \quad \textrm{ and }\quad  \qquad E_-(t)=\bigoplus_{0>\lambda(t)\in \sigma (P(t))}E_{\lambda(t)}.
\end{equation}
and let   $P^+(t):\mathbb{C}^{2n}\rightarrow E_+(t)$ be the positive eigenprojection. 

For $ \theta\in[0,1]$, we consider the homotopy $P_\theta(t)=\theta  P(t)+(1-\theta)(P^+(t)-(\Id-P^+(t)))$. Clearly   $P_1(t)=P(t)$ and $P_0(t)=P^+(t)-(\Id-P^+(t))= 2 P^+(t)-\Id$ is the conjugation operator. Being $P^+(t)$ a orthogonal projector (actually the positive eigenprojector) operator, it follows that  the matrix $P_0(t)$ only has eigenvalues $\pm 1$.

We now introduce the $\theta$-dependent one parameter family of Sturm-Liouville problems  
\begin{equation}\label{eq:c=0 for general P and A=I}
\begin{cases}
-\dfrac{d}{dt}\big(P_\theta(t)x'(t)\big)+s_0P_\theta(t)x(t)=0, & t \in [0,T]\\ x(0)=x(T), \quad  x'(0)=x'(T)
\end{cases}
\end{equation}
and the corresponding $\theta$-dependent one parameter family of Hamiltonian systems defined by
\begin{equation}\label{eq:Hamiltonian system with sP and homotopy for A=I}
\begin{cases}
	z'(t)=JB_{\theta,s_0}(t)z(t),& t \in [0,T]\\
	z(0)=z(T)
\end{cases}
\end{equation}
where $B_{\theta,s_0}(t)\=\begin{bmatrix} P_\theta^{-1}(t) &0\\0&s_0P_\theta(t) \end{bmatrix}$. We denote by $\psi_{\theta,s_0}$  the fundamental solution of Hamiltonian system defined by Equation~\eqref{eq:Hamiltonian system with sP and homotopy for A=I}.  By the homotopy property of the $\iCLM$-index, we have
\begin{multline}\label{eq:boxhomotopy}
\iCLM\big(\Delta, \Gr(\psi_{0,s_0}(t)),t\in[0,T]\big)+\iCLM\big(\Delta, \Gr(\psi_{\theta,s_0}(T)),\theta\in[0,1]\big)\\
=\iCLM\big(\Delta, \Gr(\psi_{\theta,s_0}(0)),\theta\in[0,1]\big)+\iCLM\big(\Delta, \Gr(\psi_{1,s_0}(t)),t\in[0,T]\big).
\end{multline}
Since $\psi_{\theta,s_0}(0)\equiv \Id$ for every $\theta \in [0,1]$, then $\iCLM(\Delta, \Gr(\psi_{\theta,s_0}(0)),\theta\in[0,1])=0$. 

By arguing precisely as in the proof of 
Corollary \ref{thm:non-degenerate-s_0-operatori}, it follows  that the one parameter family of (self-adjoint in $L^2$ unbounded Fredholm operators) defined by 
\[
-\dfrac{d}{dt}\left(P_\theta(t)\dfrac{d}{dt}\right)+s_0 P_\theta(T): H^2_p \subset L^2([0,T], \R^n)\to L^2([0,T], \R^n), \quad \textrm{ for } \  \theta \in [0,1]
\]
where $H^2_p\=\Set{u \in W^{2,2}([0,T],\R^n)| u(0)=u(T), u'(0)= u' (T)}$ is non-degenerate for every $\theta\in[0,1]$  and 
for large enough $s_0$.  

By using \cite[Theorem 2.5]{HS09} we get that  $\iCLM(\Delta, \Gr(\psi_{\theta,s_0}(T)),\theta\in[0,1])=0$ and then  by Equation~\eqref{eq:boxhomotopy}, we immediately get that 
 \begin{equation}\label{eq:sono-uguali}
\iCLM(\Delta, \Gr(\psi_{0,s_0}(t)),t\in[0,T])=\iCLM(\Delta, \Gr(\psi_{1,s_0}(t)),t\in[0,T]).
\end{equation}
{\underline Step 2}. Through the previous homotopy we reduced the computation to the case of a Sturm-Liouville operator having principal symbol given by a conjugation operator (in particular constant spectrum given by $\{-1,1\}$. Thus,  without leading in generalities,  we can assume that $P(t)$ is a constant symmetric matrix  with spectrum given by $\{-1,1\}$ and, as before, let $E_{\pm 1}(t)$ be the eigenspaces corresponding to the eigenvalues $\pm1$ of $P(t)$, respectively. 

Then, for each $t \in [0,T]$, there exists a 
continuous path of  orthogonal matrices 
\[
t \mapsto M(t):E_{\pm1}(t)\rightarrow E_{\pm1}(0)
\]
such that pointwise we have $P(t)=\trasp{M}(t)P(0)M(t)$ and $M(0)=\Id$. In fact, by the compactness of $[0,T]$, we only need to construct such path locally.
 Let  $\pi_{\pm}(t)$ be the orthogonal projections on $E_{\pm1}(t)$. 
 There is $\delta>0$ such that $\rk \big(\pi_+(0)\pi_+(t)\big)=k$ for $t\in (-\delta, \delta)$ where $\rk$ denotes the rank.
 Using the polar decomposition, we get that  $\pi_+(0)\pi_+(t)=S_tG_t$  such that  $S_t\in \SO(n)$ and $G_t$ is a positive semidefinite matrix.
 Then we have 
 \[E_{+1}(t)^\perp= \ker\pi_+(t)= \ker \pi_+(0)\pi_+(t)=\ker S_tG_t=\ker G_t\] and
 \[ 
 E_{+1}(0)= \im \pi_+(0)= \im \pi_+(0)\pi_+(t)=\im S_tG_t=\im S_t .
 \]
 Being $G_t$ self-adjoint, we get $\im G_t =(\ker G_t)^\perp =E_{+1}(t)$.
It follows that $S_t|_{E_{+1}(t)}=E_{+1}(0)$.

  We recall that $P(0)=P(T)$ (since $L$ is $T$-periodic) and hence $P(T)=\trasp{M}(T)P(0)M(T)=P(0)$.\footnote{
We observe that, if we have a continuous (even smooth) path of symmetric matrices $t \mapsto A_t$ we cannot  claim that  there exists a continuous family $t \mapsto Q_t$ of invertible matrices such that  $A_t= Q^{-1}_t\Delta_t Q_t$ where  $t \mapsto \Delta_t$ are diagonal. (Cfr. for further details, Kato selection's theorem \cite[Thm. II.5.4 and  Thm. II.6.8]{Kat80}). The regularity on the eigenprojections get lost when different eigenvalues path, for a certain specific value of the parameter, both coincide which cannot be the case for the path $t \mapsto P_0(t)$.}

For $\alpha \in [0,1]$, we set $A_\alpha=M(T)^{-1}M(\alpha T)$ and $P_\alpha(t)=\trasp{M}(\alpha t)P(0)M(\alpha t)$.  Thus we have \begin{equation}
\trasp{A}_\alpha P(0)A_\alpha=\trasp{M(\alpha T)}\traspm{M(T)}P(0)M(T)^{-1}M(\alpha T)=\trasp{M(\alpha T)}P(0)M(\alpha T)=:P_\alpha (T)
\end{equation}
and $A_1=\Id$. 
We now consider the following $\alpha$-dependent one parameter family of Sturm-Liouville problems
\begin{equation}\label{eq:c=0 for general P and A=I-NUOVO}
\begin{cases}
-\dfrac{d}{dt}\big(P(t)x'(t)\big)+s_0P(t)x(t)=0, & t \in [0,T]\\ x(0)=A_\alpha x(T), \quad x'(0)=A_\alpha x'(T)
\end{cases}
\end{equation}
and the corresponding one parameter family of Hamiltonian systems
\begin{equation}\label{eq:Hamiltonian system with sP and homotopy for A=I-NUOVO}
\begin{cases}
	z'(t)=JB_{s_0}(t)z(t),& t \in [0,T]\\
	z(0)=A_{\alpha d} z(T)
\end{cases}
\end{equation}
where, analogous to what denoted above, we set 
$A_{\alpha d}\=\begin{bmatrix}A_{\alpha}&0\\0&A_{\alpha}\end{bmatrix}$. 
By using the homotopy property of the Maslov index, once again, we get 
\begin{multline}\label{eq:general P A=I, eq 1}
\iCLM(\Gr(\trasp{A}_{0 d}), \Gr(\psi_{0,s_0}(t)),t\in[0,T])+\iCLM(\Gr(\trasp{A}_{\alpha d}), \Gr(\psi_{\alpha,s_0}(T)),\alpha\in[0,1])\\
=\iCLM(\Gr(\trasp{A}_{\alpha d}), \Gr(\psi_{\alpha,s_0}(0)),\alpha\in[0,1])+\iCLM(\Gr(\trasp{A}_{1 d}), \Gr(\psi_{1,s_0}(t)),t\in[0,T])
\end{multline}
and if $s_0$ is chosen larger enough,  we get that 
\begin{equation}\label{eq:general P A=I, eq 2}
\iCLM(\Gr(\trasp{A}_{\alpha d}), \Gr(\psi_{\alpha,s_0}(T)),\alpha\in[0,1])=0.
\end{equation}
 Since for every $\alpha \in [0,1]$ it holds that  $\psi_{\alpha,s_0}(0)\equiv \Id$, then   we have
 \begin{equation}
 \iCLM(\Gr(\trasp{A}_{\alpha d}), \Gr(\psi_{\alpha,s_0}(0)),\alpha\in[0,1])=\iCLM(\Gr(\trasp{A}_{\alpha d}), \Delta,\alpha\in[0,1]).
 \end{equation}
 Consider a two-family $\gamma_{\alpha ,t}=\begin{pmatrix}
 t \trasp{A}_{\alpha}&0\\0& 1/t \trasp{A}_{\alpha} \\ 
 \end{pmatrix}
 ,t\in [1,2]$.
 By using homotopy  property of Maslov index, we have
 \begin{equation}
 \begin{aligned}
 \iCLM(\Gr(\trasp{A}_{\alpha d}), \Delta,\alpha\in[0,1])=& \iCLM(\Gr(\gamma_{0,t}),\Delta,t\in [1,2])+\iCLM(\Gr(\gamma_{\alpha, 2}),\Delta,\alpha \in [0,1])\\
 &-\iCLM(\Gr(\gamma_{1,t}),\Delta,t\in [1,2])
 \end{aligned}
 \end{equation}
 Note that $\Gr(\gamma_{\alpha, 2})\cap \Delta =\{0\},\alpha\in [0,1]$.
 So we have $\iCLM(\Gr(\gamma_{\alpha, 2}),\Delta,\alpha \in [0,1])=0$.
 
We can use crossing form to calculate $\iCLM(\Gr(\gamma_{0,t}),\Delta,t\in [1,2])$.
Note that 
\[Q=-J(\frac{d}{dt}\gamma_{0,t})\gamma_{0,t}^{-1}\Big|_{t=1}=\begin{pmatrix}
0&-I_n\\-I_n&0
\end{pmatrix}.\]
Note that for any nonempty subspace $V:=\Graph(\gamma_{0t}) \cap \Delta$ of $\Delta$, $Q|_V$ is non-degenerate and $n_+(Q|_V)=n_-(Q|_V)$. 
Since  $\Gr(\gamma_{\alpha,t})\cap \Delta =\{0\}$ for each $1<t\leq 2,\alpha\in [0,1]$, by Equation \eqref{eq:iclm-crossings} we have \begin{equation}
\iCLM(\Gr(\gamma_{0,t}),\Delta,t\in [1,2])=n_+(Q|_{\Gr(\gamma_{0,1})\cap \Delta})=1/2 \dim(\Gr(\gamma_{0,1})\cap \Delta)=\dim \ker (\trasp{A_0}-I).
\end{equation}
Similarly we have $\iCLM(\Gr(\gamma_{1,t}),\Delta,t\in [1,2])=n$ , since $A_1=I_n$.
Then we can conclude that

\begin{equation}\label{eq:general P A=I, eq 3}
\iCLM(\Gr(\trasp{A}_{\alpha d}), \Gr(\psi_{\alpha,s_0}(0)),\alpha\in[0,1])=\dim\ker(\trasp{A}_0-\Id)-n.
\end{equation}
Since for $\alpha=0$, the matrix $P_0(t)=P(0)$ is constant, by applying Lemma \ref{lem:computation of special index p constant}, we get 
\begin{multline}\label{eq:general P A=I, eq 4}
\iCLM(\Gr(\trasp{A}_{0 d}), \Gr(\psi_{0,s_0}(t)),t\in[0,T])=\dim\ker(\trasp{A}_0-\Id)\\ \iCLM(\Gr(\trasp{A}_{1 d}), \Gr(\psi_{1,s_0}(t)),t\in[0,T])= \dim\ker(\trasp{A}_0-\Id).
\end{multline}
By using Equation~\eqref{eq:general P A=I, eq 1}, \eqref{eq:general P A=I, eq 2}, \eqref{eq:general P A=I, eq 3} and finally  Equation~\eqref{eq:general P A=I, eq 4} and after recalling that $A_1=\Id$,  we get 
\begin{equation}\label{eq:general P A=I, eq 5}
\iCLM(\Delta, \Gr(A_{1d}\psi_{1,s_0}(t)),t\in[0,T])=\iCLM(\Gr(\trasp{A}_{1 d}), \Gr(\psi_{1,s_0}(t)),t\in[0,T])=n.
\end{equation}
This concludes the proof.
\end{proof}
The following step is to generalize to the non-constant principal symbol Sturm-Liouville operators, the first statement given  in Lemma \ref{lem:computation of special index p constant} involving not periodic boundary conditions.
\begin{rem}
 As the previous lemmas appearing before, also Proposition \ref{thm:index-formula-for-general-P-2} is a homotopically oriented proof. In this  case involving general boundary conditions the construction of the admissible homotopy  has to take care both of the principal symbol as well as of the boundary conditions since along the through homotopy the first relation given in Equation \eqref{eq:condition-of-P-after-trivialization} has to be satisfied. 
\end{rem}
\begin{prop}\label{thm:index-formula-for-general-P-2}
Under the notation above, the following equality holds
\begin{equation}\label{eq:formula-p-generale}
\iCLM(\Gr(\trasp{A}_d),\Gr(\psi_{0,s_0}(t)), t \in [0,T])=\dim\ker(A-\Id).
\end{equation}
where $\psi_{c,s}$ denotes the fundamental solution of the Hamiltonian system given in Equation~\eqref{eq:hs-general}.
\end{prop}
\begin{proof}

For, we start to consider a path of  orthogonal matrices  $t \mapsto C(t)$   such that $C(0)=\Id$ and $  C(T)=A^{-1}$ and we define $P_\epsilon(t)=\trasp{C(\epsilon t)}P(t)C(\epsilon t), \epsilon\in[0,1]$. Clearly
 $P_0(t)=P(t)$; moreover  
 \[
 P(T)=\trasp{A}P(0)A,\qquad   P_1(T)=\traspm{A}P(T)A^{-1}=P(0), \qquad \textrm{ and  } \quad P_1(0)=P(0),
 \]
  then we have $P_1(0)=P_1(T)$. By setting $A_\epsilon=AC(\epsilon T)$, one can easily check that  $A_0=A, A_1=\Id$ and
 \begin{equation}
\trasp{A}_\epsilon P_\epsilon(0)A_\epsilon=\trasp{C(\epsilon T)}\trasp{A}\trasp{C(0)}P(0)C(0)AC(\epsilon T)=\trasp{C( \epsilon T)}P(T)C(\epsilon T)=P_\epsilon(T).
 \end{equation}
Let us now consider the system 
\begin{equation}\label{eq:c=0 for general P and A=I-epsilon}
\begin{cases}
-\dfrac{d}{dt}\big(P_\epsilon(t)x'(t)\big)+s_0P_\epsilon(t)x(t)=0, &t \in [0,T]\\x(0)=A_\epsilon x(T), \qquad  x'(0)=A_\epsilon x'(T)
\end{cases}
\end{equation}
and the corresponding $\epsilon$-dependent one parameter family of Hamiltonian systems defined by
\begin{equation}\label{eq:crazy}
\begin{cases}
	z'(t)=JB_{\epsilon,s_0}(t)z(t),& t \in [0,T]\\
	z(0)=A_{\epsilon d}z(T)
\end{cases}
\end{equation}
where $B_{\epsilon,s_0}(t)\=\begin{bmatrix} P_\epsilon^{-1}(t) &0\\0&s_0P_\epsilon(t) \end{bmatrix}$. We denote by $\psi_{\epsilon,s_0}$  the fundamental solution of Hamiltonian system defined by Equation~\eqref{eq:crazy}.  By the homotopy property of the $\iCLM$ index we have
\begin{multline}\label{eq:general P and A, eq 1}
\iCLM(\Gr(\trasp{A}_{0 d}), \Gr(\psi_{0,s_0}(t)),t\in[0,T])+\iCLM(\Gr(\trasp{A}_{\epsilon d}), \Gr(\psi_{\epsilon,s_0}(T)),\epsilon \in[0,1])\\
=\iCLM(\Gr(\trasp{A}_{\epsilon d}), \Gr(\psi_{\epsilon,s_0}(0)),\epsilon\in[0,1])+\iCLM(\Gr(\trasp{A}_{1 d}), \Gr(\psi_{1,s_0}(t)),t\in[0,T]).
\end{multline}
By the  discussion similar to the one given in the proof of Proposition \ref{thm:index-formula-for-general-P-1}, we get that $\iCLM$-index
\begin{equation}\label{eq:general P and A, eq 2}
\iCLM(\Gr(\trasp{A}_{\epsilon d}), \Gr(\psi_{\epsilon,s_0}(T)),\epsilon\in[0,1])=0.
\end{equation}
 Since $\psi_{\epsilon,s_0}(0)\equiv \Id$, then by \cite[Proposition 4.34]{Wu16}, we have
\begin{equation}\label{eq:general P and A, eq 3}
\begin{aligned}
\iCLM(\Gr(\trasp{A}_{\epsilon d}), \Gr(\psi_{\epsilon,s_0}(0)),\epsilon\in[0,1])&=\iCLM(\Gr(\trasp{A}_{\epsilon d}), \Delta,\epsilon\in[0,1])\\
&=\dfrac{1}{2}(\dim(\Gr(\trasp{A}_{0 d})\cap \Delta)-\dim(\Gr(\trasp{A}_{1 d})\cap \Delta))\\
&=\dfrac{1}{2}(2\dim\ker(\trasp{A}-\Id)-2n)=\dim\ker(\trasp{A}-\Id)-n.
\end{aligned}
\end{equation}
Since $P_1(0)=P_1(T)$ and $A_1=\Id$, then by Proposition \ref{thm:index-formula-for-general-P-1}  we get
\begin{equation}\label{eq:general P and A, eq 4}
\iCLM(\Gr(\trasp{A}_{1 d}), \Gr(\psi_{0,s_0}(t)),t\in[0,T])=n.
\end{equation}
Now by equations \eqref{eq:general P and A, eq 1}, \eqref{eq:general P and A, eq 2}, \eqref{eq:general P and A, eq 3} and \eqref{eq:general P and A, eq 4}, we have
\begin{multline}\label{eq:general P and A, eq 5}
\iCLM(\Gr(\trasp{A}_d), \Gr(\psi_{1,s_0}(t)),t\in[0,T])=\iCLM(\Gr(\trasp{A}_{0 d}), \Gr(\psi_{1,s_0}(t)),t\in[0,T])\\=\dim\ker(\trasp{A}-\Id).
\end{multline}
This concludes the proof,  since $\dim\ker(\trasp{A}-\Id)=\dim\ker(A-\Id)$.
\end{proof}


\subsection{Proof of Theorem \ref{thm:instability theorem for non-autonomous case}}

We prove only the (contrapositive of) the first  statement in   Theorem \ref{thm:instability theorem for non-autonomous case}, being the second completely analogous.  Thus, we aim to prove that  
\[
\textrm{ if } x  \textrm{ is orientation preserving and linearly stable } \quad \Rightarrow\quad  \ispec(x) +n \textrm{ is even. }
\]
First of all, it is well-known by \cite[Theorem 2.5]{HS09}, that 
 \begin{multline}\label{eq:equation 1 for prove p general}
\ispec(x)=-\iCLM(\Delta,\Gr(A_d\psi_{1,s}(T)), s\in[0,s_0])\\ =-\iCLM(\Gr(\trasp{A}_d), \Gr(\psi_{1,s}(T)), s\in[0,s_0])
\end{multline}
and by  the path additivity and homotopy properties of the $\iCLM$-index also that 
  \begin{equation}\label{eq:equation 2 for prove p general}
  \begin{aligned}
\igeo(x)&+\iCLM(\Delta,\Gr(A_d\psi_{1,s}(T)), s\in[0,s_0])\\
&=\iCLM(\Delta,\Gr(A_d\psi_{1,s}(0)), s\in[0,s_0])+\iCLM(\Delta,\Gr(A_d\psi_{1,s_0}(t)), t\in[0,T]).
\end{aligned}
\end{equation}
 Since $\psi_{1,s}(0)\equiv \Id$, then $\iCLM(\Delta,\Gr(A_d\psi_{1,s}(0)), s\in[0,s_0])=0$. Summing up Equation~\eqref{eq:equation 1 for prove p general} and Equation~\eqref{eq:equation 2 for prove p general}, we infer that 
  \begin{equation}\label{eq:equation 3 for prove p general}
\igeo(x)=\ispec(x)+\iCLM(\Delta,\Gr(A_d\psi_{1,s_0}(t)), t\in[0,T]).
\end{equation}
By using once again the path additivity and the homotopy properties of the $\iCLM$-index we get
  \begin{multline}\label{eq:equation 4 for prove p general}
\iCLM(\Delta,\Gr(A_d\psi_{0,s_0}(t)), t\in[0,T])+\iCLM(\Delta,\Gr(A_d\psi_{c,s_0}(T)),c\in[0,1])\\
=\iCLM(\Delta,\Gr(A_d\psi_{c,s_0}(0)), c\in[0,1])+\iCLM(\Delta,\Gr(A_d\psi_{1,s_0}(t)), t\in[0,T]).
\end{multline}
 Since $\psi_{c,s_0}(0)\equiv \Id$, then $\iCLM(\Delta,Gr(A_d\psi_{c,s_0}(0)), c\in[0,1])=0$. Arguing as in  Lemma \ref{thm:non-degenerate-s_0-forme}, it follows that, for $s_0$ large enough,  $\mathcal{A}_{c,s_0}$ is non-degenerate and by \cite[Theorem 2.5]{HS09}, we have
 \begin{equation}\label{eq:equation 5 for prove p general}
\spfl(\mathcal{A}_{c,s_0},c\in[0,1])=-\iCLM(\Delta,\Gr(A_d\psi_{c,s_0}(T)), c\in[0,1])=0.
\end{equation}
 By Proposition \ref{thm:index-formula-for-general-P-2}, it follows that 
 \begin{equation}
 \iCLM(\Delta,\Gr(A_d\psi_{0,s_0}(t)), t\in[0,T])=\iCLM(\Gr(\trasp{A}_d), \Gr(\psi_{0,s_0}(t)), t\in[0,T])=\dim\ker(A-\Id)
 \end{equation}
 Summing up the previous arguments,  Equation~\eqref{eq:equation 3 for prove p general} reduces to 
 \begin{equation}\label{eq:equation 7 for prove p general}
\igeo(x)=\ispec(x)+\dim\ker(A-\Id).
\end{equation}
Being $x$  orientation  preserving (by assumption), then $\det A=1$ and being  $A$ also  orthogonal, then we get that 
\begin{itemize}
\item $n \textrm{ even } \Rightarrow \dim\ker(A-\Id) \textrm{ even }$;
\item $n \textrm{ odd } \Rightarrow \dim\ker(A-\Id) \textrm{ odd }$.
\end{itemize}
So, in both cases  we have $n-\dim\ker(A-\Id)$ is even.

Now, if $x$ is linear stable, then by taking into account  Lemma \ref{lem:instability by maslov index}, we get that $\igeo(x)$ is even and by Equation~\eqref{eq:equation 7 for prove p general} $\ispec(x)+\dim\ker(A-\Id)$ is even. Since
\begin{equation}\label{eq:ultima-1}
	n+\ispec(x)=(n-\dim\ker(A-\Id))+(\ispec(x)+\dim\ker(A-\Id)
\end{equation}
and being the (RHS) of Equation~\eqref{eq:ultima-1} sum of two even integers, we conclude that it is  even. This concludes the proof. 
\qedhere


\appendix

\appendix
%
\section{A symplectic excursion on the Maslov index}\label{sec:Maslov}

The purpose of this Section is to provide the  symplectic  preliminaries used in the paper.  In Subsection \ref{subsec:Maslovindexpath} the main 
properties  of the intersection number for curves of Lagrangian subspaces with respect to a  distinguished one are collected and the {\em (relative) Maslov index\/} is 
defined.  Our basic references are \cite{PPT04,CLM94, GPP04, RS93,MPP05,MPP07}.

\subsection{A quick recap on the $\iCLM$-index}\label{subsec:Maslovindexpath}

Given a $2n$-dimensional (real) symplectic space $(V,\omega)$, a {\em 
Lagrangian 
subspace\/} of $V$ is an $n$-dimensional subspace $L \subset V$ such that $L = 
L^\omega$ where $L^\omega$ denotes the {\em symplectic orthogonal\/}, i.e. the 
orthogonal 
with respect to the symplectic structure. 
We denote by $ \Lagr= \Lagr(V,\omega)$ the {\em Lagrangian Grassmannian of 
$(V,\omega)$\/}, namely the set of all Lagrangian subspaces of $(V, \omega)$
\[
\Lagr(V,\omega)\=\Set{L \subset V| L= L^{\omega}}.
\]
It is well-known that $\Lagr(V,\omega)$ is a manifold. For each $L_0 \in \Lagr$, 
let 
\[
\Lagr^k(L_0) \= \Set{L \in \Lagr(V,\omega) | \dim\big(L \cap L_0\big) =k } 
\qquad k=0,\dots,n.
\]
Each $\Lagr^k(L_0)$ is a real compact, connected submanifold of codimension 
$k(k+1)/2$. The topological closure 
of $\Lagr^1(L_0)$  is the {\em Maslov cycle\/} that can be also described as 
follows
\[
 \Sigma(L_0)\= \bigcup_{k=1}^n \Lagr^k(L_0)
\]
The top-stratum $\Lagr^1(L_0)$ is co-oriented meaning that it has a 
transverse orientation. To be 
more precise, for each $L \in \Lagr^1(L_0)$, the path of Lagrangian subspaces 
$(-\delta, \delta) \mapsto e^{tJ} L$ cross $\Lagr^1(L_0)$ transversally, and as 
$t$ increases the path points to the transverse direction. Thus  the Maslov cycle is two-sidedly embedded in 
$\Lagr(V,\omega)$. Based on the topological properties of the Lagrangian 
Grassmannian manifold, 
it is possible to define a fixed endpoints homotopy invariant called {\em Maslov 
index\/}.

\begin{defn}\label{def:Maslov-index}
Let $L_0 \in \Lagr(V,\omega)$ and let $\ell:[0,1] \to \Lagr(V, \omega)$ be a 
continuous path. We 
define the {\em Maslov index\/} $\iCLM$ as follows:
\[
 \iCLM(L_0, \ell(t); t \in[a,b])\= \left[e^{-\varepsilon J}\, \ell(t): 
\Sigma(L_0)\right]
\]
where the right hand-side denotes the intersection number and $0 < \varepsilon 
<<1$.
\end{defn}
For further reference we refer the interested reader to \cite{CLM94} and references therein. 
\begin{rem}
 It is worth noticing that for $\varepsilon>0$ small enough, the Lagrangian 
subspaces 
 $e^{-\varepsilon J} \ell(a)$ and $e^{-\varepsilon J} \ell(b)$ are off the 
singular cycle. 
\end{rem}
One efficient way to compute the Maslov index, was introduced by authors in 
\cite{RS93} via 
crossing forms. Let $\ell$ be a $\mathscr C^1$-curve of Lagrangian subspaces 
such that 
$\ell(0)= L$ and let $W$ be a fixed Lagrangian subspace transversal to $L$. For 
$v \in L$ and 
small enough $t$, let $w(t) \in W$ be such that $v+w(t) \in \ell(t)$.  Then the 
form 
\[
 Q(v)= \dfrac{d}{dt}\Big\vert_{t=0} \omega \big(v, w(t)\big)
\]
is independent on the choice of $W$. A {\em crossing instant\/} for $\ell$ is an 
instant $t \in [a,b]$ 
such that $\ell(t)$ intersects $W$ nontrivially. At each crossing instant, we 
define the 
crossing form as 
\[
 \Gamma\big(\ell(t), W, t \big)= Q|_{\ell(t)\cap W}.
\]
A crossing is termed {\em regular\/} if the crossing form is non-degenerate. If 
$\ell$ is regular meaning that 
it has only regular crossings, then the Maslov index is equal to 
\begin{equation}\label{eq:iclm-crossings}
 \iCLM\big(W, \ell(t); t \in [a,b]\big) = \coiMor\big(\Gamma(\ell(a), W; a)\big)+ 
\sum_{a<t<b} 
 \sgn\big(\Gamma(\ell(t), W; t\big)- \iMor\big(\Gamma(\ell(b), W; b\big)
\end{equation}
where the summation runs over all crossings $t \in (a,b)$ and $\coiMor, \iMor$ 
are the dimensions  of 
the positive and negative spectral spaces, respectively and $\sgn\= 
\coiMor-\iMor$ is the  signature. 
(We refer the interested reader to \cite{LZ00} and \cite[Equation (2.15)]{HS09}). 
We close this section by 
recalling some useful 
properties of the Maslov index. \\
\begin{itemize}
\item[]{\bf Property I (Reparametrization invariance)\/}. Let $\psi:[a,b] \to 
[c,d]$ be a 
continuous and piecewise smooth function with $\psi(a)=c$ and $\psi(b)=d$, then 
\[
 \iCLM\big(W, \ell(t)\big)= \iCLM(W, \ell(\psi(t))\big). 
\]
\item[] {\bf Property II (Homotopy invariance with respect to the ends)\/}. For 
any $s \in [0,1]$, 
let $s\mapsto \ell(s,\cdot)$ be a continuous family of Lagrangian paths 
parametrised on $[a,b]$ and 
such that $\dim\big(\ell(s,a)\cap W\big)$ and $\dim\big(\ell(s,b)\cap W\big)$ 
are constants, then 
\[
 \iCLM\big(W, \ell(0,t);t \in [a,b]\big)=\iCLM\big(W, \ell(1,t); t \in 
[a,b]\big).
\]
\item[]{\bf Property III (Path additivity)\/}. If $a<c<b$, then
\[
 \iCLM\big(W, \ell(t);t \in [a,b]\big)=\iCLM\big(W, \ell(t); t \in [a,c]\big)+
 \iCLM\big(W, \ell(t); t \in [c,b]\big) 
\]
\item[]{\bf Property IV (Symplectic invariance)\/}. Let $\Phi:[a,b] \to \Sp(2n, 
\R)$. Then 
\[
 \iCLM\big(W, \ell(t);t \in [a,b]\big)= \iCLM\big(\Phi(t)W, \Phi(t)\ell(t); t 
\in [a,b]\big).
\]
\end{itemize}


\subsection{The symplectic group and the $\iomega{1}$-index}

In the standard symplectic space $(\R^{2n}, \omega)$ we denote by $J$ the standard symplectic matrix defined by $J=\begin{bmatrix} 0&-\Id\\ \Id &0\end{bmatrix}$. The symplectic form $\omega$ can be  represented with respect to the Euclidean product $\langle\cdot, \cdot\rangle$ by $J$ as follows $\omega(z_1,z_2)=\langle J z_1,z_2\rangle$ for every $z_1, z_2 \in \R^{2n}$.  We consider the 1-codimensional (algebraic) subvariety 
\[
\Sp(2n, \R)^{0}\=\{M\in \Sp(2n, \R)| \det(M-\Id)=0\} \subset \Sp(2n, \R)
\] 
and let us define
\[
\Sp(2n,\R)^{*}=\Sp(2n,\R)\backslash \Sp(2n, \R)^{0}=\Sp(2n,
\R)^{+}\cup
\Sp(2n, \R)^{-}
\]
where 
\begin{multline}
\Sp(2n,\R)^+\=\{M\in \Sp(2n,\R)| \det(M-\Id)>0\} \quad \textrm{ and }\\
\Sp(2n,\R)^- \=\{M\in \Sp(2n,\R)
|\det(M-\Id)<0\}.
\end{multline}

For any $M\in \Sp(2n,\R)^{0}$, $\Sp(2n,\R)^{0}$ is
co-oriented at the point $M$
 by choosing  as positive direction the direction determined by
 $\frac{d}{dt}Me^{tJ}|_{t=0}$ with $t\geq0$ sufficiently small.  We recall that $\Sp(2n,\R)^{+}$ and
$\Sp(2n,\R)^{-}$ are two path connected
 components of $\Sp(2n,\R)^{*}$
  which are simple connected in $\Sp(2n,\R)$. (For the proof of these facts we refer, for instance,  the interested reader to \cite[pag.58-59]{Lon02} and references therein).  Following authors in \cite[Definition 2.1]{LZ00} we start by recalling the following definition.
 \begin{defn}\label{def:Maslov-index-ok}
Let $\psi:[a,b]\rightarrow\Sp(2n,\R)$ be a continuous path. Then there exists an $\varepsilon>0$ such that for every $\theta\in[-\varepsilon,\varepsilon]\setminus\{0\}$, the matrices $\psi(a)e^{J\theta}$ and $\psi(b)e^{J\theta}$  lying both out of $\Sp(2n,\R)^0$ . We define the {\em $\iomega{1}$-index\/} or the {\em Maslov-type index\/} as follows
\begin{equation}
\iomega{1}(\psi)\=[e^{-J\varepsilon}\psi:\Sp(2n,\R)^{0}]
\end{equation}
where the (RHS) denotes the intersection number between the perturbed path $t\mapsto e^{-J\varepsilon} \psi(t)$ with the singular cycle $\Sp(2n,\R)^0$. 
\end{defn}
Through the parity of the $\iomega{1}$-index it is possible to locate  endpoints of the perturbed symplectic path  $t\mapsto e^{-J\varepsilon} \psi(t)$.  
 \begin{lem}\label{lem:parity property}{\bf (\cite[Lemma 5.2.6]{Lon02})\/}
 Let $\psi:[a,b] \to \Sp(2n, \R)$ be a continuous path.  The  following characterization holds
 \begin{itemize}
 \item[] $\iomega{1}(\psi)$ is even $\iff$ both the endpoints
 $e^{-\varepsilon J}\psi(a)$ and $e^{-\varepsilon J}\psi(b)$ lie in $\Sp(2n, \R)^+$ or
in $\Sp(2n, \R)^-$.
 \end{itemize}
 \end{lem}
We close this section with a rather technical  result which will be used in the proof of the main instability criterion. 
\begin{lem}{\bf (\cite[Lemma 3.2]{HS10}.
)\/}\label{prop:how to know in which component}
 Let $\psi:[a,b] \to \Sp(2n, \R)$ be a continuous symplectic path such that $\psi(0)$ is linearly stable.
 \begin{enumerate}
  \item  If $1 \notin \sigma\big(\psi(a)\big)$ then there exists $\varepsilon >0$
sufficiently small such that
  $\psi(s) \in \Sp(2n, \R)^+$ for $|s| \in (0, \varepsilon)$.
  \item We assume that $\dim \ker \big(\psi(a)-\Id\big)=m$ and
  $\trasp{\psi(a)}J \psi'(a)\vert_V$ is non-singular for $V\= \psi^{-1}(a) \R^{2m}$. If
  $\ind\big({\trasp{\psi(a)}J \psi'(a)\vert_V}\big)$ is even [resp. odd] then there exists $\delta>0$
  sufficiently small such that  $\psi(s) \in \Sp(2n, \R)^+$
  [resp. $\psi(s) \in \Sp(2n, \R)^-$] for $|s| \in (0, \delta)$.
  \end{enumerate}
\end{lem}
\begin{rem}
Knowing that $M \in \Sp(2n,\R)^0$, without any further information, it is not possible a priori to locate in which path connected components of $\Sp(2n, \R)^*$ is located the perturbed  matrix $e^{\pm \delta J}M$ for arbitrarily small positive  $\delta$. However if $M$ is linearly stable, we get the following result.
\end{rem} 
\begin{lem}\label{lem:linear stable is in positive side}
 Let $M\in \Sp(2n, \R) $ be a linearly stable symplectic matrix (meaning that $\spec{M} \subset \U$ and $M$ is diagonalizable). 
 Then, there exists $\delta >0$ sufficiently small such that $e^{\pm \delta J}M \in \Sp(2n,\R)^+$.
\end{lem}
\begin{proof}
Let us consider the  symplectic path pointwise defined by
 $M(\theta)\= e^{-\theta J}M$. By a direct computation we get that
\[
 \trasp{M(\theta)} J \dfrac{d}{d\theta} M(\theta)\Big\vert_{\theta=0} =
\trasp{M}M.
\]
We observe that $\trasp{M}M$ is symmetric and positive semi-definite; moreover since $M$ invertible it follows that
$\trasp{M}M$ is actually positive definite. Thus, in particular,
$\iiindex(\trasp{M}M)=0$. By invoking Lemma \ref{prop:how to know in which component} it follows
that there exists $\delta >0$ such that $M(\pm \delta) \in \Sp(2n, \R)^+ $. This concludes the proof.
\end{proof}


\section{On the Spectral Flow for bounded selfadjoint Fredholm operators}\label{sec:spectral-flow}
%
Let $\mathcal W, \mathcal H$ be  real separable Hilbert spaces with a dense 
and 
continuous inclusion $\mathcal W \hookrightarrow \mathcal H$.
\begin{note}
We denote by  
$\mathcal{B}(\mathcal W,\mathcal H)$ the Banach  space of all linear bounded 
operators (if $\mathcal W=\mathcal H$ we use the shorthand notation 
$\mathcal{B}(\mathcal H)$); by $\mathcal{B}^{sa}(\mathcal W, \mathcal H)$ we 
denote the set of all  bounded selfadjoint operators when regarded as operators 
on  $\mathcal H$ and finally $\mathcal{BF}^{sa}(\mathcal W, \mathcal H)$ 
denotes 
the set of all bounded selfadjoint Fredholm operators and we recall that an 
operator $T \in \mathcal{B}^{sa}(\mathcal W, \mathcal H)$ is Fredholm if and 
only if its kernel is finite dimensional and its image is closed. 
\end{note}
For $T \in\mathcal{B}(\mathcal W, \mathcal H)$ we recall that the {\em spectrum 
\/} of $T$ is $\sigma(T)\= \Set{\lambda \in \C| T-\lambda I \text{ is not 
invertible}}$ and that $\sigma(T)$ is decomposed into the {\em essential 
spectrum\/} and the {\em discrete spectrum\/} defined respectively as 
$\sigma_{ess}(T) \= \Set{\lambda \in \C| T-\lambda I \notin 
\mathcal{BF}(\mathcal W, \mathcal H)}$ and $\sigma_{disc}(T)\= \sigma(T) 
\setminus \sigma_{ess}(T)$.
It is worth noting that $\lambda \in \sigma_{disc}(T)$ if and only if it is an 
isolated point in $\sigma(T)$ and $\dim \ker (T - \lambda I)<\infty$.

Let now $T \in \mathcal{BF}^{sa}(\mathcal W,\mathcal H)$, then either $0$ is 
not 
in $\sigma(T)$ or it is in $\sigma_{disc}(T)$ (cf. \cite[Lemma 2.1]{Wat15}), 
and, as a consequence of the Spectral Decomposition Theorem (cf. \cite[Theorem 
6.17, Chapter 
III]{Kat80}), the following orthogonal decomposition holds
\[
 \mathcal W = E_-(T) \oplus \ker T \oplus E_+(T),
\]
with the property
\[
 \sigma(T) \cap(-\infty, 0)= \sigma\left(T_{E_-(T)}\right) \textrm{ and } 
 \sigma(T) \cap(0,+\infty)= \sigma\left(T_{E_+(T)}\right).
\]
\noindent
\begin{defn}\label{def:Morseindex}
Let $T \in \mathcal{BF}^{sa}(\mathcal W,\mathcal H)$. If $\dim E_-(T)<\infty$ 
(resp.  $\dim 
E_+(T)<\infty$), 
we define its {\em Morse index\/} (resp. {\em Morse co-index\/})
as the integer denoted by $\iMor(T)$  (resp. $\coiMor(T)$) and defined as:
\[
 \iMor(T) \= \dim E_-(T)\qquad \big(\textrm{resp. } \coiMor(T)\= \dim E_+(T)\big).
\]
\end{defn}
\smallskip
The space  $\mathcal{BF}^{sa}(\mathcal H)$ was intensively 
investigated by Atiyah and Singer in \cite{AS69} \footnote{%
Actually, in this reference, only skew-adjoint Fredholm operators were 
considered, but the case of bounded selfadjoint Fredholm operators presents no 
differences.} and the following important topological characterisation can be 
deduced.
\begin{prop}\label{thm:as69} (Atiyah-Singer, \cite{AS69})
The space $\mathcal{BF}^{sa}(\mathcal H)$ consists of three connected 
components:
\begin{itemize}
\item the {\em essentially positive\/} $
 \mathcal{BF}^{sa}_+(\mathcal H)\=\Set{T \in\mathcal{BF}^{sa}(\mathcal 
H)|\sigma_{ess}(T) 
  \subset (0,+\infty)}$; 
 \item the  {\em essentially negative\/} 
  $\mathcal{BF}^{sa}_-(\mathcal H)\=\Set{T \in\mathcal{BF}^{sa}(\mathcal 
H)|\sigma_{ess}(T) 
  \subset (-\infty,0)}$;
  \item the {\em strongly indefinite\/}
$ \mathcal{BF}^{sa}_*(\mathcal H)\=\mathcal{BF}^{sa}(\mathcal 
H)\setminus(\mathcal{BF}^{sa}_+(\mathcal H)
  \cup \mathcal{BF}^{sa}_-(\mathcal H)). $
  \end{itemize}
The spaces $\mathcal{BF}_+^{sa}(\mathcal H),\mathcal{BF}_-^{sa}(\mathcal H)$ 
are  contractible (actually convex), whereas $ \mathcal{BF}^{sa}_*(\mathcal 
H)$ is topological non-trivial; more precisely, $\pi_1(\mathcal{BF}^{sa}_*(\mathcal 
H))\simeq\Z.$
\end{prop}
\begin{rem}
By the definitions of the connected components of $\mathcal{BF}^{sa}$, we 
deduce 
that a bounded linear  operator is essentially positive if and only if it is a symmetric 
compact perturbation of a (bounded) positive definite selfadjoint operator. 
Analogous observation hold for essentially negative operators. 
\end{rem}
Even if in the strongly indefinite case the  Morse index as well as the  Morse co-index are meaningless, it is possible to define a sort of relative version, usually called {\em relative Morse index\/}. In order to fix our notation, we need to recall some definitions and basic facts. Our basic references are \cite{Abb01,Kat80}.

Let  $V, W$ be two  closed subspaces of $\mathcal H$ and let $P_V$ (resp. $ P_W$)  denote the orthogonal projection onto $V$ (resp. $W$).  $V, W$ are commensurable if the operator $P_V- P_W$ is compact.
\begin{defn}\label{def:commensurable-pairs}
 Let  $V$ and $W$ be two closed commensurable subspaces of $\mathcal H$. The {\em relative dimension\/} of $W$ with respect to $V$ is the integer 
\begin{equation}\label{eq:commensurable}
\dim(W,V)\= \dim (W \cap V^\perp)- \dim (W^\perp \cap V).	
\end{equation}
\end{defn}
\begin{rem}\label{rem:mi-serve}
	We observe that the spaces $W \cap V^\perp$ and  $W^\perp \cap V$  appearing in the (RHS) of Equation~\eqref{eq:commensurable} are finite dimensional because they are the spaces of fixed points of the compact operators $P_{V^\perp}P_W$ and $P_{W^\perp}P_V$, respectively. (For further details, cfr. \cite[pag.44]{Abb01}).
\end{rem}
\begin{prop}(\cite[Proposition 2.3.2]{Abb01})\label{thm:abbo2.3.2}
Let  $S,T \in \mathcal{BF}^{sa}(\mathcal H)$ be 
such that $S-T$ is a compact. Then the negative (resp. positive) eigenspaces are commensurable.  
\end{prop}
As direct consequence of Proposition \ref{thm:abbo2.3.2}, 
we are entitled to introduce the following definition.
\begin{defn}\label{def:relativeMorseindex}
We define the {\em relative Morse index\/} of an ordered pair of selfadjoint Fredholm operators  $S, T \in \mathcal{BF}^{sa}(\mathcal H)$ such that $S-T$ 
is compact, as the integer 
\begin{equation}\label{eq:relative-dimension-general}
 \irel(T,S)\= \dim\Big(E_-(S)\cap \big(E_+(T)\oplus E_0(T)\big)\Big)-\dim\Big(E_-(T)\cap \big(E_+(S)\oplus E_0(S)\big)\Big).
\end{equation}
\end{defn}
\begin{rem}
For further details on pairs of commensurable subspaces and Fredholm pairs, we refer the interested reader to \cite[Chapter 2]{Abb01}. Compare with \cite{ZL99} and \cite[Definition 2.1]{HS09}.
\end{rem}
If $S,T\in \mathcal{GL}^{sa}(\mathcal H)$, such that $S-T$ 
is compact then the relative Morse index given in Equation~\eqref{eq:relative-dimension-general}, reduces to   
\begin{equation}\label{eq:relative-dimension-special}
 \irel(T,S)\= \dim\big(E_-(S),E_-(T)\big).
\end{equation}
If $S$ is a compact perturbation of a non-negative definite 
operator $T$, so, it is essentially positive and $
\irel(T,S)=\iMor(S).$ 
\begin{rem}\label{rmk:segno-sf}
It is worth noticing that in the special case of Calkin equivalent positive isomorphisms, authors in \cite{FPR99}, defined the relative Morse index $I(T,S)$. However, their definition agrees with the one given in Definition \ref{def:relativeMorseindex} only up to the sign.  
\end{rem}

We are now in position to introduce the spectral flow. 
Given a  $\mathscr C^1$-path  $L:[a,b]\to\mathcal{BF}^{sa}(\mathcal W, \mathcal 
H)$, the spectral flow of $L$ counts the net number of eigenvalues crossing 0. 
\begin{defn}\label{def:crossing}
An instant $t_0 \in (a,b)$ is called a \emph{crossing instant} (or {\em 
crossing\/} for short) if $\ker  L_{t_0} \neq \{0\}$. The \emph{crossing form} 
at a crossing $t_0$ is the quadratic form defined by 
\[
 \Gamma( L, t_0): \ker  L_{t_0} \to \R, \ \ \Gamma( L, 
t_0)[u] \=\langle 
 \dot{ L}_{t_0} u, u \rangle_{\mathcal H},
\]
where we denoted by $\dot{L}_{t_0}$ the derivative of $L$ 
with respect to the parameter $t \in [a,b]$ at the point $t_0$.
A crossing is called \emph{regular}, if $\Gamma( L, t_0)$ is 
non-degenerate. If $t_0$ is a crossing instant for $L$, we refer to 
$m(t_0)$ the dimension of $\ker  L_{t_0}$.
\end{defn}

\begin{rem}
It is worth noticing that regular crossings are isolated, and hence, on a compact interval are in a finite number. Thus we are entitled to introducing the following definition.
\end{rem}

In the case of regular curve (namely a curve having only regular crossings)  we introduce the following Definition. 
\begin{defn}\label{def:new-spectralflow-def}
 Let  $L:[a,b]\to\mathcal{BF}^{sa}(\mathcal  H)$ be a $\mathscr 
C^1$-path and 
 we assume that it has only regular crossings. Then 
 \begin{equation}\label{eq:spectral-flow-crossings}
\spfl(L; [a,b])= \sum_{t \in (a,b)} \sgn \Gamma(L, t)- 
\iMor\big(\Gamma(L,a)\big)
+ \coiMor\big(\Gamma(L,b)\big),
\end{equation}
where the sum runs over all regular (and hence in a finite number) strictly contained in  $[a,b]$.
\end{defn}
\begin{rem}
It is worth noticing that, it is always possible to perturb the path $L$ for getting a regular path. Moreover, as consequence of the fixed end-points homotopy invariance of the spectral flow, the spectral flows of the original (unperturbed) and the perturbed  paths both coincide.
\end{rem}

\begin{defn}\label{def:positive-paths}
	The $\mathscr C^1$-path $L:[a,b]\ni t \mapsto L_t\in \mathcal{BF}^{sa}(\mathcal  H)$ is termed {\em positive\/} or {\em plus\/} path, if at each crossing instant $t_*$ the crossing form $\Gamma(L, t_*)$ is positive definite.  
\end{defn}
\begin{rem}
We observe that in the case of a positive path, each crossing is regular and in particular the total number of crossing instants on a compact interval is finite. Moreover the local contribution at each crossing to the spectral flow is given by the dimension of the intersection. Thus given a positive path $L$, the spectral flow is given by 
 \[
\spfl(L; [a,b])= \sum_{t \in (a,b)} \dim  \ker L(t)+ \dim  \ker L(b). 
\]
\end{rem}
\begin{defn}\label{def:admissible-paths-operators}
The path $L:[a,b]\to\mathcal{BF}^{sa}(\mathcal  H)$ is termed {\em admissible\/} provided it has invertible endpoints. 
\end{defn}
For paths of bounded self-adjoint Fredholm operators parametrized on $[a,b]$ which are compact perturbation of a fixed operator,  the spectral flow given in Definition \ref{def:new-spectralflow-def}, can be characterized as the relative Morse index of its endpoints. More precisely, the following result holds. 
\begin{prop}\label{thm:spfl-operatori-diff-rel-morse}
Let us consider the  path  $L: [a,b] \to\mathcal{BF}^{sa}(\mathcal  H)$  and we assume that for every $t \in [a,b]$, the operator  $L_t- L_a$ is compact.   Then
\begin{equation}\label{eq:equality-spfl-relmorse}
-\spfl(L; [a,b])=\irel(L_a, L_b).
\end{equation}
Moreover if $L_a$ is essentially positive, then we have 
\begin{equation}\label{eq:diff-Morse}
-\spfl(L; [a,b])
=\iMor (L_b)-\iMor(L_a)
\end{equation}
and if furthermore  $L_b$ is positive definite, then 
\begin{equation}
\spfl(L; [a,b])=\iMor(L_a).
\end{equation}
\end{prop} 
\begin{proof}
The proof of the equality in  Equation~\eqref{eq:equality-spfl-relmorse} is an immediate consequence of the fixed end homotopy properties of the spectral flow. For, let $\varepsilon >0 $ and let us consider the two-parameter family 
\[
L:[0,1]\times [a,b] \to \mathcal{BF}^{sa}(\mathcal  H) \textrm{ defined by } L(s,t)\= L_t+ s \,\varepsilon\, \Id. 
\]
By the homotopy property of the spectral flow, we get that 
\begin{multline}\label{eq:fff}
\spfl(L_t; t \in [a,b])\\	=  \spfl(L_a+s \varepsilon \Id, s \in [0,1]) + \spfl(L_t+\varepsilon\Id, t \in [a,b])- \spfl(L_b + s \varepsilon\Id, s \in [0,1])\\=  \spfl(L_t+\varepsilon\Id, t \in [a,b])
\end{multline}
where the last equality in  Equation~\eqref{eq:fff} is consequence if the positivity of all the involved paths.  By choosing a maybe smaller $\varepsilon>0$ the path  $t\mapsto L_t+\varepsilon\Id$ is admissible (in the sense of Definition  \ref{def:admissible-paths-operators}). The conclusion, now readily follows by applying  \cite[Proposition 3.3]{FPR99} (the minus sign appearing is due to a different choosing convention for the spectral flow. Cfr. Remark \ref{rmk:segno-sf}) to the path $t\mapsto L_t+\varepsilon\Id$. 
In order to prove the second claim, it is enough to observe that if $L_a$ is essentially positive, then $L$ is a  path entirely contained in the (path-connected component) $\mathcal{BF}_+^{sa}(\mathcal  H)$. The proof of the equality in Equation~\eqref{eq:diff-Morse} is now a direct  consequence of Equation the previous argument and \cite[Proposition 3.9]{FPR99}. The last can be deduced by Equation~\eqref{eq:diff-Morse} once observed that $\iMor(L_b)=0$. This concludes the proof. 
\end{proof}
\begin{rem}
	We observe that  a direct proof of Equation~\eqref{eq:diff-Morse} can be easily conceived as direct consequence of the homotopy properties of $\mathcal{BF}_+^{sa}(\mathcal  H)$.
\end{rem}

\begin{rem}
 We observe that the definition of spectral flow for bounded selfadjoint Fredholm operators given in Definition \ref{def:new-spectralflow-def} is slightly different from the standard definition given in literature in which only continuity is required on the regularity of the path.  (Cfr. For further details, we refer the interested reader to 
\cite{Phi96,RS95, Wat15} and 
 references therein). Actually Definition \ref{def:new-spectralflow-def}  represents an efficient way for 
 computing the spectral flow even if it requires more regularity as well as a  transversality assumption 
 (the regularity of each crossing instant). However, it is worth to mentioning that, the spectral flow is a fixed endpoints homotopy invariant and for admissible paths (meaning for paths having invertible endpoints) is a free homotopy invariant. By density arguments, we observe that a $\mathscr C^1$-path 
 always exists  in any fixed endpoints homotopy class of the original path.  
\end{rem}

\begin{rem}
 It is worth noting, as already observed by author in \cite{Wat15}, that the spectral flow can be 
 defined in the more general case of continuous 
 paths of closed unbounded selfadjoint Fredholm operators that are 
 continuous with respect to the (metric) gap-topology (cf. \cite{BLP05} and references 
 therein). However in the special case in 
 which the domain of the operators is fixed, then the closed path of unbounded 
 selfadjoint Fredholm operators can be regarded as a continuous path 
 in $\mathcal{BF}^{sa}(\mathcal W, \mathcal  H)$. Moreover  this path is also continuous 
 with respect to the aforementioned gap-metric topology.
 
 The advantage to regard the paths in  $\mathcal{BF}^{sa}(\mathcal W, \mathcal  H)$ is that the 
 theory is straightforward as in the bounded case and, clearly, it is sufficient for the applications  
 studied in the present manuscript. 
\end{rem}

\subsection{On the Spectral Flow for Fredholm quadratic forms}\label{subsec:flussospettraleforme}
%
{{
Following authors in \cite{MPP05} we are in position to discuss the spectral  flow for bounded Fredholm quadratic forms.  
Let $(\mathcal H, \langle \cdot,\cdot \rangle)$ be a real separable Hilbert space. A continuous function $q : \mathcal H \to  \R$  is called a {\em quadratic form \/} provided that there is a symmetric bilinear form  $b_q : \mathcal H \times \mathcal H \to \R$ such that $q(u) = b_q(u,u)$ for all $u \in \mathcal H$.
 The bilinear form $b_q$ is determined by $q$ through the following polarization identity:
 \[
 b_q(u,v)=\dfrac14\{ q(u+v)-q(u-v)\} \textrm{ for all } u,v \in \mathcal H. 
 \]
This implies, in particular, that the bilinear form $b_q$  is also continuous. Let us denote by $\mathcal Q(\mathcal H)$  the set of all bounded quadratic  forms on $\mathcal H$ and we observe that $\mathcal Q(\mathcal H)$ is a Banach space with the norm defined by
\[ 
\norm{q} \= \sup_{\norm{u}\leq 1} |q(u)| \textrm{ for all }  q \in \mathcal Q(\mathcal H).
\]
\begin{lem}
	Let $q\in \mathcal Q(\mathcal H)$. Then there exists a unique self–adjoint operator $T \in \mathcal B(\mathcal H)$ called the representation of $q$ with respect to $\langle \cdot,\cdot\rangle $ with the property that
	    \begin{equation}\label{eq:2.2}
	     q(u) =\langle T\,u, u \rangle  \textrm{ for all  } u \in  \mathcal H
	     \end{equation}
	     Moreover 
	     \begin{equation}\label{eq:2.3}
	     b_q(u,v)=\langle T\,u, v \rangle  \textrm{ for all  } u, v \in \mathcal H.
	     \end{equation}
	    \end{lem}
\begin{proof}
We start to observe that, from the polarization identity, it follows immediately that an operator satisfies Equation~\eqref{eq:2.2} if and only if it satisfies Equation 
	\eqref{eq:2.3}. Fix $u \in \mathcal H$. Since $b_q$ is continuous and bilinear, then the map  $v \mapsto b_q(u,v)$ is linear and continuous on $\mathcal H$. By the Riesz-Fréchét Representation Theorem, there is an element that we denote by $Tu$ having the property that $b_q(u,v)=\langle T\,u, v \rangle$ for all $v \in \mathcal H$. Since $b_q$ is symmetric, bilinear and continuous we conclude that the operator $T$ belongs to $\mathcal B(\mathcal H)$ and it is symmetric with respect to $ \langle \cdot,\cdot \rangle$.
 \end{proof}
 A quadratic form $q \in \mathcal Q(\mathcal H)$ is called {\em non-degenerate\/} provided that 
\[
b_q(u,v)=0 \textrm{ for all } v \in H \Rightarrow u=0.
\]
 \begin{defn}\label{def:Fqf}
 Under above notation, a quadratic form $q:\mathcal H\to \R$ is termed a \emph{Fredholm quadratic form\/}  if the operator $T$  given in Equation~\eqref{eq:2.2} (representing $q$ with respect to the scalar product of $\mathcal H$) is Fredholm. 
\end{defn}

\begin{rem}
The set $\mathcal{Q_F}(\mathcal H)$ is an open subset of $\mathcal Q(\mathcal H)$. We observe that the operator representing of a non-degenerate quadratic form has trivial kernel and therefore if the form is also Fredholm then such a representation belongs to $\GL(\mathcal H)$.
\end{rem}
\begin{lem}\label{thm:stable-perturbation}
	A quadratic form on $\mathcal H$ is weakly continuous if and only  its representation is a compact operator in $\mathcal H$.
\end{lem}
\begin{proof}
For the proof, we refer the interested reader, to \cite[Appendix B]{BJP14}. 
\end{proof}
We define a quadratic form $q : \mathcal H \to \R$ to be positive definite provided that, there exists $c>0$ such that 
\[
q(u) \geq c \norm{u}^2, \textrm{ for all } u \in \mathcal H.
\]
We are now entitled to  introduce the following definition which will be crucial in the whole paper. 
\begin{defn}\label{def:essentially-positive}
	A quadratic form $q : \mathcal H \to \R$ is termed {\em essentially positive\/} and we write $q\in \mathcal{Q_F}^+(\mathcal H)$, provided it is a weakly continuous 
perturbation of a positive definite quadratic (Fredholm) form.\end{defn}
It turns out that if $q\in 
\mathcal{Q_F}^+(\mathcal H)$ then its representation is and essentially  positive  selfadjoint Fredholm operator. 
\begin{rem} 
We observe that as consequence of Proposition \ref {thm:as69}, the set $\mathcal{Q_F}^+(\mathcal H)$ 	is contractible.
\end{rem}
\begin{defn}\label{def:Morse-essentially-positive}
	The {\em Morse index\/}, denoted by $\iMor(q)$, of an essentially positive Fredholm quadratic form $q: \mathcal H \to \R$ is defined as the Morse index of its representation. Thus, in symbol
	\[
	\iMor(q)= \dim E_-(T)
	\]
	where $T$ is the representation of $q$ and $E_-(T)$ denotes the negative spectral space of $T$.
\end{defn}

We are now in position to introduce the notion of {\em spectral flow\/} for (continuous) path of Fredholm quadratic forms. 
\begin{defn}\label{def:sfquadratic}
Let $q:[a,b]\rightarrow \mathcal {Q_F}(\mathcal H)$ be a continuous   path.
We define  the \emph{spectral flow of $q$}  as the  spectral 
flow of the continuous path induced by its representation; namely
\[
\spfl(q(t); t \in[a,b]) \= \spfl(L_t; t \in [a,b])
\]
where we denoted by $L_t $ the representation of q(t),  namely   $q(t)= \langle L_t \cdot, \cdot \rangle_{\mathcal H}$.
\end{defn}
Following authors in \cite{MPP05}, if a path $q : [a, b] \to \mathcal {Q_F}(\mathcal H)$ is differentiable at $t$ then the derivative $\dot q(t)$ is also a quadratic form. We shall say that the instant $t$ is a crossing point if $\ker b_q(t)\neq \{0\}$, and We shall say that the crossing instant $t$ is regular if the crossing form $\Gamma(q, t)$, defined as the restriction of the derivative $\dot q(t)$ to the subspace $\ker b_q(t)$, is non-degenerate. We observe that regular crossing points are isolated and that the property of having only regular crossing forms is generic for paths in $\mathcal {Q_F}(\mathcal H)$. 
As direct consequence of Definition \ref{def:sfquadratic} and Definition \ref{def:new-spectralflow-def}, if $q:[a,b] \to \mathcal {Q_F}(\mathcal H)$ is a $\mathscr C^1$-path of Fredholm quadratic forms having only regular crossings,  we infer that  
\begin{equation}\label{eq:sf-forme-crossings}
	\spfl(q(t); t \in[a,b])=\sum_{t \in (a,b)} \sgn \big(\Gamma(q(t))\big)- 
\iMor\big(\Gamma(q(a))\big)
+ \coiMor\big(\Gamma(q(b))\big) 
\end{equation} 
where the sum runs over all regular (and hence in a finite number) strictly contained in  $[a,b]$.
\begin{defn}\label{def:admissible-path-forms}
A path $q :[a, b] \to \mathcal {Q_F}(\mathcal H)$ is called {\em admissible\/} provided that it has non-degenerate endpoints.
\end{defn}
Thus by Definition \ref{def:admissible-path-forms} and Equation~\eqref{eq:sf-forme-crossings}, we get that for an admissible $\mathscr C^1$-path in  $\mathcal {Q_F}(\mathcal H)$ having only regular crossings, the spectral reduces to 
\begin{equation}\label{eq:sf-forme-crossings-admi}
	\spfl(q(t); t \in[a,b])=\sum_{t \in (a,b)} \sgn \big(\Gamma(q(t))\big)
\end{equation}
where the sum runs over all regular (and hence in a finite number) strictly contained in  $[a,b]$.
(Cfr. \cite[Theorem 4.1]{FPR99}). 
\begin{prop}\label{thm:sf-differenza-morse}
Let $q:[a,b] \to 	 \mathcal {Q_F}(\mathcal H)$ be a path of essentially positive Fredholm quadratic forms. Then 
\begin{equation} \label{eq:sf-diffMorse}
\spfl(q(t);t\in[a,b]) = \iMor\big(q(a)\big)-\iMor\big(q(b)\big).
\end{equation}
Furthermore, if $q(b)$ is positive definite, then 
$\spfl(q(t);t \in [a,b]) = \iMor(q(a))$. 
\end{prop}
\begin{proof}
Let $T:[a,b] \to \mathcal{BF}^{sa}( \mathcal  H)$ be the path of selfadjoint Fredholm operators that represents the path of Fredholm  quadratic forms $q$. As already observed, for each $t \in [a,b]$, each  $T(t)$ is an essentially positive Fredholm operator. The result follows by Proposition \ref{thm:spfl-operatori-diff-rel-morse} to the path $t \mapsto T_t$. \\
The second claim follows readily by observing that if $q(b)$ is essentially positive and non-degenerate, then it is positive and hence the Morse index of its representation vanishes.
This concludes the proof.
\end{proof}
\begin{defn}\label{def:generalized-Fredholm}
A {\em generalized   family of Fredholm quadratic forms\/}
parameterized by an interval is a smooth function   $q\colon
{\mathcal H}\to \R ,$ where ${\mathcal H}$ is a Hilbert bundle
over $[a,b]$ and $q$ is such that its restriction $q_t$  to the
fiber ${\mathcal H}_t$  over $t$ is a Fredholm quadratic form. If
$q_a$ and $q_b$ are non  degenerate, we define the spectral flow
$\spfl(q) = \spfl (q,[a,b])$  of such a family $q$ by choosing a
trivialization $$M \colon [a,b] \times {\mathcal H}_a \to {\mathcal
H}$$ and defining
\begin{equation} \label{sflow2}
\spfl(q) = \spfl (\tilde{q},[a,b])
\end{equation}
where $\tilde{q}(t)u =q_t (M_tu).$
\end{defn}
It follows from cogredience property that the right hand side of Equation~\eqref{sflow2} is independent of the choice of the trivialization.
Moreover all of the above properties hold true in this more
general case, including the calculation of the spectral flow
through a non degenerate crossing point,  provided we substitute the usual derivative with
the intrinsic derivative of a bundle map.


\vspace{1cm}
\noindent
\textsc{Prof. Alessandro Portaluri}\\
DISAFA\\
Università degli Studi di Torino\\
Largo Paolo Braccini 2 \\
10095 Grugliasco, Torino\\
Italy\\
Website: \url{aportaluri.wordpress.com}\\
E-mail: \email{alessandro.portaluri@unito.it}

\vspace{1cm}
\noindent
\textsc{Prof. Li Wu}\\
Department of Mathematics\\
Shandong University\\
Jinan, Shandong, 250100\\
The People's Republic of China \\
China\\
E-mail:\email{nankai.wuli@gmail.com}

\vspace{1cm}
\noindent
\textsc{Dr. Ran Yang}\\
School of Science\\
East China  University of Technology\\
Nanchang, Jiangxi, 330013\\
The People's Republic of China \\
China\\
E-mail:\email{yangran2019@ecit.cn}

\begin{thebibliography}{99}



\bibitem[Abb01]{Abb01}
{\sc Abbondandolo, A.}
\newblock Morse Theory for Hamiltonian Systems
\newblock Chapman and Hall/CRC, New York (2001)

\bibitem[AF07]{AF07}
{\sc Abbondandolo, Alberto; Figalli, Alessio }
\newblock High action orbits for Tonelli Lagrangians and superlinear Hamiltonians on compact configuration spaces.
\newblock  J. Differential Equations 234 (2007), no. 2, 626--653.

\bibitem[APS08]{APS08}
{\sc Abbondandolo, Alberto; Portaluri, Alessandro; Schwarz, Matthias}
\newblock The homology of path spaces and
 Floer homology with conormal boundary conditions.
 \newblock J. Fixed Point Theory Appl. 4 (2008), no. 2, 263--293.

%

\bibitem[AS09]{AS09}
{\sc Abbondandolo, Alberto; Schwarz, Matthias}
\newblock A smooth pseudo-gradient for the Lagrangian action functional. 
\newblock Adv. Nonlinear Stud. 9 (2009), no. 4, 597--623.


\bibitem[AS69]{AS69}
{\sc Atiyah, M. F.; Singer, I. M.}
\newblock Index theory for skew-adjoint Fredholm operators. 
\newblock Inst. Hautes Études Sci. Publ. Math. No. 37 1969 5--26.



\bibitem[BHPT19]{BHPT19}
{\sc Barutello, Vivina; Hu, Xijun; Portaluri, Alessandro; Terracini Susanna}
\newblock An Index theory for asymptotic motions under singular potentials
\newblock Preprint available \url{https://arxiv.org/abs/1705.01291}

\bibitem[BJP16]{BJP16}
{\sc Barutello, Vivina; Jadanza, Riccardo D.; Portaluri, Alessandro}
\newblock Morse index and linear stability of the Lagrangian circular orbit
in a three-body-type problem via index theory.
\newblock Arch. Ration. Mech. Anal. 219 (2016), no. 1, 387--444.

\bibitem[BJP14]{BJP14}
{\sc Barutello, Vivina L.; Jadanza, Riccardo D.; Portaluri, Alessandro}
\newblock Linear instability
of relative equilibria for n-body problems in the plane.
\newblock J. Differential Equations 257 (2014), no. 6, 1773--1813.

\bibitem[BLP05]{BLP05}
Booss-Bavnbek, Bernhelm; Lesch, Matthias; Phillips, John
Unbounded Fredholm operators and spectral flow. Canad. J. Math. 57 (2005),no.2, 225–250.
\bibitem[BT15]{BT15}
{\sc Bolotin, S. V.; Treshch\"ev, D. V.}
\newblock Anti-integrable limit. (Russian)
\newblock  Uspekhi Mat. Nauk 70 (2015), no. 6(426), 3--62;
translation in Russian Math. Surveys 70 (2015), no. 6, 975--1030.


\bibitem[CLM94]{CLM94}
{\sc Cappell, Sylvain E.; Lee, Ronnie; Miller, Edward Y.}
\newblock On the Maslov index.
\newblock Comm. Pure Appl. Math. 47 (1994), no. 2, 121--186.

\bibitem[Fat08]{Fat08}
{\sc Fathi, Albert} 
\newblock Weak KAM theorem in Lagrangian dynamics. Preliminary version. Number 10
\newblock  (2008) 


\bibitem[FPR99]{FPR99}
{\sc Fitzpatrick, P. M.; Pejsachowicz, J.; Recht, L.}
\newblock Spectral flow and bifurcation of critical points of 
strongly-indefinite functionals. I. 
General theory. 
\newblock J. Funct. Anal. 162 (1999), no. 1, 52--95.  


\bibitem[GPP04]{GPP04}
{\sc Giambò, Roberto; Piccione, Paolo; Portaluri, Alessandro}
\newblock Computation of the Maslov index and the spectral flow via partial
signatures
\newblock C. R. Math. Acad. Sci. Paris 338 (2004), no. 5, 397--402.



\bibitem[HP17]{HP17}
{\sc Hu, Xijun; Portaluri, Alessandro}
\newblock  Index theory for heteroclinic orbits of Hamiltonian systems. 
\newblock Calc. Var. Partial Differential Equations 56 (2017), no. 6, Art. 167, 24 pp.
\newblock Preprint available on \url{https://arxiv.org/abs/1703.03908}

\bibitem[HP19]{HP19}
{\sc Hu, Xijun; Portaluri, Alessandro} 
\newblock Bifurcation of heteroclinic orbits via an index theory.
\newblock Math. Z. 292 (2019), no. 1-2, 705--723.
\newblock  Preprint available on \url{https://arxiv.org/abs/1704.06806}

\bibitem[HPY17]{HPY17}
{\sc Hu, Xijun; Portaluri, Alessandro; Yang Ran}
\newblock A dihedral Bott-type iteration formula and stability of symmetric periodic orbits
\newblock  Preprint available on \url{https://arxiv.org/pdf/1705.09173.pdf}

\bibitem[HPY19]{HPY19}
{\sc Hu, Xijun; Portaluri, Alessandro; Yang Ran}
\newblock  Instability of  semi-Riemannian closed geodesics.
\newblock To apper in Nonlinearity.
\newblock  Preprint available on \url{https://arxiv.org/pdf/1706.07619.pdf}	

\bibitem[HS09]{HS09}
{\sc Hu, Xijun; Sun, Shanzhong}
\newblock Index and stability of symmetric periodic
orbits in Hamiltonian systems with application to figure-eight orbit.
\newblock Comm. Math. Phys. 290 (2009), no. 2, 737--777.

\bibitem[HS10]{HS10}
{\sc Hu, XiJun; Sun, ShanZhong}
\newblock Morse index and the stability of closed geodesics.
\newblock Sci. China Math. 53 (2010), no. 5, 1207--1212.


\bibitem[Kat80]{Kat80}
{\sc Tosio Kato}
\newblock Perturbation Theory for linear operators.
\newblock  Grundlehren der Mathematischen Wissenschaften, 132,
Springer-Verlag (1980).
\bibitem[Lon02]{Lon02}
{\sc  Long, Yiming}
\newblock Index theory for symplectic paths with applications.
\newblock Progress in Mathematics, 207. Birkh\"auser Verlag, Basel, 2002.


\bibitem[ZL99]{ZL99}
{\sc  Long, Yiming; Zhu, Chaofeng}
\newblock Maslov-type index theory for symplectic paths and spectral flow. I.
\newblock Chinese Ann. Math. Ser. B 20 (1999), no. 4, 413--424.



\bibitem[LZ00]{LZ00}
{\sc Long, Yiming; Zhu, Chaofeng}
\newblock Maslov-type index theory for symplectic paths
and spectral flow. II
\newblock  Chinese Ann. Math. Ser. B 21 (2000), no. 1, 89--108.
%


\bibitem[MPW17]{MPW17}
{\sc Marchesi, Giacomo; Portaluri, Alessandro; Waterstraat Nils}
\newblock Not every conjugate point of a semi-Riemannian geodesic is a bifurcation point.
\newblock Differential Integral Equations 31 (2018), no. 11-12, 871--880. 
\newblock Preprint available on \url{https://arxiv.org/abs/1703.10483}

\bibitem[MPP05]{MPP05}
{\sc Musso, Monica; Pejsachowicz, Jacobo; Portaluri, Alessandro,}
\newblock A Morse index theorem for perturbed geodesics on semi-Riemannian
manifolds.
\newblock Topol. Methods Nonlinear Anal. 25 (2005), no. 1, 69--99.


\bibitem[MPP07]{MPP07}
{\sc Musso, Monica; Pejsachowicz, Jacobo; Portaluri, Alessandro}
\newblock Morse index and bifurcation of $p$-geodesics on
semi Riemannian manifolds.
\newblock ESAIM Control Optim. Calc. Var. 13 (2007), no. 3, 598--621.

\bibitem[Phi96]{Phi96}
{\sc Phillips J.}
\newblock Self-adjoint Fredholm operators and spectral flow.
\newblock Canadian Mathematical Bulletin 39(1996), no. 4, 460--467.



\bibitem[PPT04]{PPT04}
{\sc Piccione, Paolo; Portaluri, Alessandro; Tausk, Daniel V.}
\newblock Spectral flow, Maslov index and
bifurcation of semi-Riemannian geodesics.
\newblock Ann. Global Anal. Geom. 25 (2004), no. 2, 121--149.

\bibitem[Por08]{Por08}
{\sc Portaluri, Alessandro}
\newblock Maslov index for Hamiltonian systems.
\newblock Electron. J. Differential Equations 2008, No. 09.

\bibitem[PW16]{PW16}
{\sc Portaluri, Alessandro; Waterstraat Nils}
\newblock A $K$-theoretical Invariant and Bifurcation for Homoclinics of Hamiltonian Systems.
\newblock  J. Fixed Point Theory Appl. 19 (2017), no. 1, 833–851.
\newblock Preprint available on \url{https://arxiv.org/abs/1605.08402}


\bibitem[RS93]{RS93}
{\sc Robbin, Joel; Salamon, Dietmar}
\newblock The Maslov index for paths.
\newblock Topology 32 (1993), no. 4, 827--844.

\bibitem[RS95]{RS95}
{\sc Robbin, Joel; Salamon, Dietmar}
\newblock The spectral flow and the Maslov index.
\newblock Bull. London Math. Soc. 27 (1995), no. 1, 1--33.


\bibitem[Tres88]{Tres88}
{\sc Treschev, D. V.}
\newblock  The connection between the Morse index of a
closed geodesic and its stability. (Russian)
\newblock Trudy Sem. Vektor. Tenzor. Anal. No. 23 (1988),
175--189.

\bibitem[Wat15]{Wat15}
{\sc Waterstraat, Nils}
\newblock Spectral flow, crossing forms and homoclinics of Hamiltonian
systems.
\newblock Proc. Lond. Math. Soc. (3) 111 (2015), no. 2, 275--304.


\bibitem[Wu16]{Wu16}
{\sc Li, Wu}
\newblock The Calculation of Maslov-type index and Hörmander index in
weak symplectic Banach space.
\newblock Doctoral Degree Thesis, Nankai University, September, 2016





\end{thebibliography}
\end{document}